\documentclass{amsart}

\usepackage{amscd}
\usepackage{amssymb}
\usepackage{cite}
\usepackage{enumerate}
\usepackage{amsmath}
\usepackage[usenames, dvipsnames]{color}
\usepackage{hyperref}

\def\w*lim{\mathop{\mbox{\textup{w*-lim}}}}

\usepackage{mathrsfs}

\newtheorem{theorem}{\sc \textbf{Theorem}}[section]
\newtheorem{lemma}[theorem]{\sc \textbf{Lemma}}
\newtheorem{cor}[theorem]{\sc \textbf{Corollary}}
\newtheorem{prop}[theorem]{\sc \textbf{Proposition}}
\newtheorem{proposition}[theorem]{\sc \textbf{Proposition}}
\newtheorem{rmk}[theorem]{\sc \textbf{Remark}}
\newtheorem{definition}[theorem]{\sc \textbf{Definition}}

\newtheorem{lem}[theorem]{\sc \textbf{Lemma}}

\newcommand{\cA}{{\mathcal A}}
\newcommand{\cB}{{\mathcal B}}

\newcommand{\cE}{{\mathcal E}}

\newcommand{\cH}{{\mathcal H}}

\newcommand{\cM}{{\mathcal M}}
\newcommand{\cN}{{\mathcal N}}

\newcommand{\cP}{{\mathcal P}}

\newcommand{\cU}{{\mathcal U}}

\newcommand{\cX}{{\mathcal X}}
\newcommand{\cY}{{\mathcal Y}}

\newcommand{\bR}{{\mathbb R}}

\newtheorem{example}[theorem]{Example}

\begin{document}

\title[Alberti--Uhlmann problem]{Alberti--Uhlmann problem on  Hardy--Littlewood--P\'{o}lya majorization
}

\dedicatory{Dedication to the 90th birthday of Professor Armin Uhlmann}

\author[J. Huang]{J. Huang}
\address[Jinghao Huang]{School of Mathematics and Statistics, University of New South Wales, Kensington, 2052, NSW, Australia,
\emph{E-mail~:}{\tt jinghao.huang@unsw.edu.au}}

\author[F. Sukochev]{F. Sukochev}
\address[Fedor A. Sukochev]{School of Mathematics and Statistics, University of New South Wales, Kensington, 2052, NSW, Australia  \emph{E-mail~:} {\tt f.sukochev@unsw.edu.au}
}

\subjclass[2010]{46L51; 46L10; 46E30; 15B51. \hfill Version~: \today}



\thanks{Fedor Sukochev was supported by the Australian Research Council  (FL170100052).}

\begin{abstract}
 We fully describe the doubly stochastic orbit of a self-adjoint  element in the noncommutative $L_1$-space affiliated with a semifinite von Neumann algebra, which answers a problem posed by Alberti and Uhlmann \cite{AU} in the 1980s,
 extending  several results in the literature.
 It follows further from our methods that, for any $\sigma$-finite  von Neumann algebra $\cM$ equipped a semifinite infinite faithful normal trace $\tau$,  there exists a self-adjoint operator $y\in L_1(\cM,\tau)$ such that the doubly stochastic orbit of $y$ does not  coincide with the orbit of $y$ in the sense of Hardy--Littlewood--P\'{o}lya,
 which confirms  a conjecture by Hiai \cite{Hiai}.
 However, we show that Hiai's conjecture fails for non-$\sigma$-finite von Neumann algebras.
The main  result of the present paper also answers
 the (noncommutative) infinite counterparts of
 problems due to  Luxemburg \cite{Luxemburg} and Ryff \cite{Ryff63} in the 1960s.
\end{abstract}
\maketitle

\tableofcontents

\section{Introduction}

The partial order $\prec$  for real $n$-vectors  (customarily termed  Hardy--Littlewood--P\'{o}lya majorization) was introduced in the early 20th century, by
Muirhead \cite{Muirhead}, Lorenz\cite{Lorenz}, Dalton\cite{Dalton} and Schur\cite{Schur}
(see also a fundamental monograph by Hardy, Littlewood and P\'{o}lya\cite{HLP}).
Hardy--Littlewood--P\'{o}lya majorization
plays a vital role in the study of function spaces, Banach
lattices and interpolation theory and  has important applications in stochastic
analysis, numerical analysis, geometric inequalities, matrix theory, statistical theory, optimization  and economic theory~\cite{MOA}.

The starting point of this paper is a well-known result in classic analysis due to  Hardy, Littlewood and P\'{o}lya:
$ x\prec y\in \bR^n$ if and only if $x$ belongs to the convex hull $\Omega(y)$ of the set of all permutations of $y$ \cite{HLP} (see also \cite[p.10]{MOA} and \cite{Rado}), i.e., the convex hull of $\{ Py : P \mbox{ is  a permutation matrix}\}$.
Alternatively, one can say that  $\{ Py : P \mbox{ is  a permutation matrix}\}$ is the set of all extreme points of  $\{x\in \bR^n: x\prec y\}$.
This result has an important noncommutative counterpart, which states that,  in the setting of $n\times n$ matrices,
 a Hermitian   matrix $A$ belongs to  the doubly stochastic orbit of another Hermitian  matrix $B$  ($A$ is said to be more chaotic than $B$ for physical interpretation \cite{AU}) if and only if the vector  $\lambda (A)$ of eigenvalues of  $A$ is majorized by that of $B$ in the sense of Hardy--Littlewood--P\'{o}lya  (denoted by $A\prec B$) \cite{Rado}.
 This, in turn,   is closely related to the Birkhoff--von Neumann theorem
 identifying extreme points of doubly stochastic matrices with permutation matrices\cite{Birkhoff46}.
The (Hardy--Littlewood--P\'{o}lya) orbit $\Omega (B):= \{A \mbox{ is a Hermitian $n\times n$ matrix}:A\prec B\}$
can be described in terms of
 unitary mixing and convex functions.
We record  these well-known  results  as follows. 
\begin{theorem}
\cite[Theorem 2.2]{AU}\cite{Ando}\label{AU}
Let $A,B$ be Hermitian $n\times n$ matrices.
The following conditions are equivalent:
\begin{enumerate}
  \item[(a)] $A$ is in the doubly stochastic orbit of $B$, that is, $A=TB$ for some doubly stochastic operator $T$ (i.e., a positive linear map which preserves the trace and the identity);
  \item[(b)] $A$ is majorized by $B$, that is $A\prec B$;
  \item[(c)]  $A$ is in the convex hull of elements $C$ which are unitarily equivalent to $B$, i.e.,
      $C$ and $B$ have the same eigenvalues;
  \item[(d)]$A$ is in the convex hull of elements which are unitarily equivalent to $B$ and commute with $A$;
  \item[(e)] $\sum_{k=1}^n ( \lambda (A)-t)_+  \le \sum_{k=1}^n (\lambda(B)-t)_+  $ for any $t\in \bR$ and ${\rm Tr}(A)= {\rm Tr}(B)$;
  \item[(f)] for any convex function $f$ on the real axis, we have ${\rm Tr}(f(A))\le {\rm Tr}(f(B))$\footnote{In \cite[Theorem 2.2]{AU}, the authors use any concave function on the real axis, ${\rm Tr}(f(A))\ge {\rm Tr}(f(B))$''. Actually, it is equivalent with condition ($f$) by  taking $-f$.}.
\end{enumerate}
Hence, ${\rm Tr}$ stands for the standard trace of a matrix.
\end{theorem}
Alberti and Uhlmann developed a unitary mixing theory which is motivated by physical problems related to  irreversible processes, the classification of mixed states (in the sense of Gibbs and von Neumann),  and general diffusions\cite[p.8]{AU}.
One of the central problems considered by
Alberti and Uhlmann  is
 as follows \cite[Chapter 3, p.58]{AU}:
\begin{quote}
``We,\dots, may ask ourselves how to formulate a variant of
 Theorem~\ref{AU}  in the setting of von Neumann algebras?''
\end{quote}

This paper addresses a number of seemingly disparate open problems, and we shall see that they are in fact deeply intertwined.
These questions are loosely concentrated around Alberti--Uhlmann attempts \cite{AU} to extend majorization theory for general von Neumann algebras and Hiai--Nakamura attempts \cite{Hiai,HN2,Hiai92} to extend the Hardy--Littlewood--P\'{o}lya majorization to the setting of strictly  infinite semifinite von Neumann algebras.
We explain below the connection of these two themes with classical themes in analysis and algebra.

The notion of majorization in the setting of Lebesgue measurable functions on $(0,1)$ is due to Hardy, Littlewood and P\'{o}lya \cite{HLP29} (see also \cite{Ryff63,Luxemburg}).
As we mention above,
$ x\prec y\in \bR^n$ if and only if $x$ belongs to the convex hull $\Omega(y)$ of the set of all permutations of $y$ \cite{HLP,MOA,Rado}.
In 1967, Luxemburg (see \cite[Problem 1]{Luxemburg}) asked for a continuous counterpart of this result:
\emph{how to describe   extreme points of the set $\Omega(f)$ of all  elements majorized by an integrable function $f$ on a finite measure space?}
The special case for finite Lebesgue measure spaces was resolved by Ryff~\cite{R,Ryff,Ryff68}, who showed that the set of all extreme point
 of the orbit $\Omega(f)$ coincides with the set of all measurable functions having the same decreasing rearrangement with $f$.
However,
in full generality, Luxemburg's problem was answered only recently in  \cite{DHS}.
In particular, this provides a   variant of the equivalence between (b) and (c) in Theorem \ref{AU} for arbitrary finite measure spaces.
Ryff \cite{Ryff63}
also asked
for the equivalence between (a) and (b) in Theorem \ref{AU} in the setting of finite Lebesgue  measure spaces, i.e.,
{\it  whether
the doubly stochastic orbit and the orbit in the sense of Hardy--Littlewood--P\'{o}lya coincide,}
 and answered affirmatively later in \cite{R}.
 This result was extended by Day \cite{Day} to   arbitrary  finite measure spaces.

In the particular case of finite von Neumann algebras, Alberti--Uhlmann problem can be viewed as the noncommutative counterpart of problems  due to Luxemburg and Ryff,
which
has been widely studied during the past decades.
Recall that the Hardy--Littlewood--P\'{o}lya majorization in finite matrices is defined in terms of vectors of eigenvalues of finite matrices.
The eigenvalue function of  a self-adjoint operator, the analogue of vector of  eigenvalues of a Hermitian finite matrices,  in the noncommutative $L_1$-space affiliated with a finite von Neumann algebra was introduced by  Murray and von Neumann \cite{MN36} (also by Grothendieck~\cite{Gro},  by Ov\v{c}innikov \cite{Ov} and by Petz\cite{P}).
In terms of eigenvalues functions, Kamei~\cite{Kamei83,Kamei} and Hiai~\cite{Hiai} defined
Hardy--Littlewood--P\'{o}lya majorization in this setting. 
Since then,
several mathematicians have made contribution to
 Alberti--Uhlmann problem for  the positive core of a finite von Neumann algebra \cite{Hiai,HN2,HN1,CKSa,Sukochev,S85}.
In the noncommutative $L_1$-space affiliated with
a   finite von Neumann algebra,
the extreme points of the orbit (in the sense of Hardy--Littlewood--P\'{o}lya) of a self-adjoint operator   were fully characterized  in \cite{DHS}.
In Section~\ref{AUF},
we complement results in \cite{DHS,Hiai,HN1}
  by proving the equivalence between (a) and (b) in Theorem~\ref{AU} in this setting   finite von Neumann algebra  (see Theorem~\ref{th:f}), which
provides a resolution of the noncommutative counterpart of  Ryff's question for doubly stochastic orbits in this setting.

In the full generality, Alberti--Uhlmann problem can be viewed as a noncommutative and infinite version of
Luxemburg's problem  and Ryff's problem at the same time.
Due to the complicated nature of an infinite  von Neumann algebra,
Alberti--Uhlmann problem (even in the setting of infinite measure spaces) is substantially more difficult than Luxemburg's problem  and Ryff's problem.
From now on, we focus on Alberti--Uhlmann problem in the strictly infinite setting.
This setting however is plagued by numerous technical difficulties and below, we explain their origin and earlier attempts to overcome those.

In 1946, Birkhoff \cite{Birkhoff46} showed that
the extreme points of the set of all  doubly stochastic matrices are permutation matrices
and asked for an extension of this result to the infinite-dimensional case \cite[Problem 111]{Birkhoff48}.
Since then, doubly stochastic operators  have been  actively studied   by many mathematicians such as Isbell \cite{Isbell,Isbell55}, Hiai\cite{Hiai,HN1}, Kaftal and Weiss \cite{KW11,KW08,KW11} (see also \cite{Weiss}), Kendall \cite{Kendall}, Komiya \cite{Komiya},  Rattray and Peck \cite{RP},  R\'{e}v\'{e}sz \cite{Revesz},  Sakamaki and Takahashi \cite{ST}, and Tregub\cite{Tregub81,Tregub86}.
In particular, the extreme points of  infinite stochastic  matrices are permutation matrices \cite{Kendall,Revesz,Isbell,Isbell55}.
However, the description of extreme points of doubly stochastic operator on $l_\infty$ (or a von Neumann algebra) is still unclear (see \cite{Hiai,Tregub81,Tregub86} and references therein).
The notion of  Hardy--Littlewood--P\'{o}lya majorization for infinite sequences has been discussed by various authors (see \cite[p.25]{MOA} and references therein).
A natural question in this area is about the equivalence between
conditions (a) and (b) in Theorem \ref{AU} in the infinite setting, that is,
\emph{does  there exist an  element $y\in l_\infty$ such that the doubly stochastic orbit of $y$ does not coincide with its orbit in the sense of Hardy--Littlewood--P\'{o}lya?} (see next paragraph for a more general question by Hiai \cite{Hiai}).
As far as we know,
there are no prior known examples showing that these two conditions are not equivalent in the general infinite  setting.
Nevertheless, there are results
giving partial answers to this question under additional assumptions such as the vectors are positive and decreasing \cite{Komiya,ST,KW11}.
In Section~\ref{s:h}, we provide a negative answer to  this question in the full generality.

In the setting of finite von Neumann algebras, Hiai\cite[Theorem 4.7 (1)]{Hiai} showed that the doubly stochastic orbit of a positive operator coincides with its Hardy--Littlewood--P\'{o}lay orbit.
Due to the lack of natural meaning of Hardy--Littlewood--P\'{o}lya majorization in the infinite case,
there are several different  extensions of this notion  in the literature (see e.g. \cite{NRS,Skoufranis,Hiai,HN1,AU,KS}).
 Hiai and Nakamura \cite{HN1} (see also \cite[p. 25]{MOA} and \cite{NRS,Skoufranis})
  provided a natural definition of Hardy--Littlewood--P\'{o}lya majorization in terms of eigenvalue functions
 in the setting of infinite von Neumann algebras:
\begin{quote}
  Let $L_1(\cM,\tau)$ be the noncommutative $L_1$-space affiliated with a semifinite von Neumann algebra equipped with a faithful normal semifinite trace $\tau$.
  Let $x,y\in L_1(\cM,\tau)$ be self-adjoint.
  Then,  $x$ is said to be majorized by $y$ (denoted by $x\prec y$) in the sense of Hardy--Littlewood--P\'{o}lya  if
  $x_+\prec \prec y_+$, $x_-\prec \prec y_-$ and $\tau(x)=\tau(y)$ \cite[p. 7]{HN1}.
  Here, $\prec \prec $ stands for the Hardy--Littlewood--P\'{o}lya submajorization.
\end{quote}
In the finite setting, the above definition coincides with the classic definition of Hardy--Littlewood--P\'{o}lya majorization\cite{HN1}.
It turns out that doubly stochastic operators behave quite differently  in the infinite setting from the finite setting and Hiai couldn't extend  \cite[Theorem 4.7 (1)]{Hiai}  to  inifinite von Neumann algebras.
He
conjectured that  \cite[p. 40]{Hiai}):
\begin{quote}
Because the set of all doubly stochastic operators on a von Neumann algbra is no longer BW-compact\cite{Arveson} when the trace is infinite, the assumption that the trace is finite  in \cite[Theorem 4.7 (1)]{Hiai}  seems essential.
\end{quote}
Hiai didn't provide an example showing that the ``finite trace'' condition is sharp.
In Section \ref{s:h}, we provide several examples confirming Hiai's conjecture.
Moreover, we show that
for a semifinite von Neumann algebra $\cM$,
the doubly stochastic orbit of $y$ coincides with its Hardy--Littlewood--P\'{o}lya orbit
for any self-adjoint element $y\in L_1(\cM,\tau)$
 if and only if $\cM$ is not $\sigma$-finite equipped with a faithful normal semifinite infinite trace  $\tau$.
In particular, this result shows that
Hiai's conjecture is true when $\cM$ is  $\sigma$-finite and $\tau$ is infinite but
 this  conjecture fails for non-$\sigma$-finite infinite von Neumann algebras.


The following theorem is the main result  of the present paper, which  answers (extensions of) problems due to Luxemburg \cite{Luxemburg}, due to Ryff\cite{Ryff63}, due to Hiai\cite{Hiai}, and due to Alberti and Uhlmann \cite{AU} in the  setting  of infinite von Neumann algebras (all notations are introduced in Section \ref{prel}).
It extends numerous existing results in the literature such as \cite{DHS,Ryff63,Ryff,R,Ryff68,ST,Komiya,Hiai,HN1,KW08,KW}.
 \begin{theorem}\label{main}
 Let $\cM$ be a semifinite  von Neumann algebra equipped with a semifinite infinite faithful normal trace $\tau$.
 For any $y\in L_1(\cM,\tau)_h$ and  any $x\in L_1(\cM,\tau)_h$, the following conditions are equivalent:
 \begin{enumerate}
  \item[(a).] there exists two   semifinite von Neumann algebra  $(\cA,\tau_1)$ and $(\cB,\tau_2)$ with $\cM\subset \cA , \cB$ and the restriction of $\tau_1$ (and $\tau_2$) on $\cM$ coincides with $\tau$,  and
   there exists a (normal) doubly stochastic  operator $\varphi:\cA\to \cB$ such that $\varphi(y)=x$.
   \item[(b).] $x\in \Omega(y)$, i.e.,  $x\prec y$ (see Proposition \ref{majtraceequ} and  Theorems \ref{stoinfi} and \ref{stopureinf});
     \item[(c).]$\tau(x) =\tau(y)$, $\tau((x-r)_+)\le \tau((y-r)_+)$ and
  $\tau((-x-r)_+)\le \tau((-y-r)_+)$ for all $r\in \bR$ (see Proposition \ref{treq});
   \item[(d).]  $\tau(f(x)) \le \tau(f(y))$ for all convex function $f$ on $\bR$ with $f(0)=0$ such that $f(x)$ and $f(y)$ are integrable (see Proposition \ref{btod});

 \end{enumerate}
 In addition, the extreme points of $\Omega(y)$
 are those elements $ x \in L_1(\cM,\tau)_h$ satisfying that (see Theorem \ref{th:ext}) for $x_+$ and $y_+$ and for  any $t\in  (0,\infty)$, one of the following options holds:
       \begin{enumerate}
  \item[(i).]   $\lambda(t;x_+)= \lambda(t;y_+)$;
  \item [(ii).]
   $\lambda(t;x_+) \ne \lambda(t;y_+)$ with  the spectral projection $ E^{x_+} \{\lambda (t;x_+)  \} $ being an atom in $\cM$ and $$\int_{ \{s;\lambda (s;x_+)=\lambda (t;x_+)\}}  \lambda (s;y_+)ds    =  \lambda(t ;x_+)     \tau(E^{x_+} (\{\lambda (t;x_+)\}))  , $$
\end{enumerate}
and,  for $x_-$ and $y_-$, for any $t\in  (0,\infty )$, one of the following options holds:
       \begin{enumerate}
  \item[(i).]   $\lambda(t;x_-)=  \lambda(t;y_-)$;
  \item [(ii).]
   $\lambda(t;x_-) \ne \lambda(t;y_-)$ with  the spectral projection $ E^{x_-} \{\lambda (t;x_-)  \} $ being an atom in $\cM$ and $$\int_{ \{s;\lambda (s;x_-)=\lambda (t;x_-)\}}  \lambda (s;y_-)ds    =  \lambda(t ;x_-)     \tau(E^{x_-} (\{\lambda (t;x_-)\}))  .  $$
\end{enumerate}
Moreover, if $\cM$ is a semifinite infinite factor, then $x\in \overline{\{uyu^*:u\in \cU(\cM)\}}^{\|\cdot\|_1}$ \cite[Theorem 3.5]{HN1};
$\cA$ and $\cB$ can be chosen to be  $\cM$ if and only if $\cM$ is not $\sigma$-finite (see Theorems \ref{sigmanot} and  \ref{stopureinf}).
 \end{theorem}



The authors would like to thank  Jean-Christophe Bourin,  Thomas Scheckter and Dmitriy Zanin for their helpful discussion and  useful  comments.

\section{Preliminaries}\label{prel}
In this section,
we recall some notions of the theory of noncommutative integration.
In what follows,  $\cH$ is a  Hilbert space and $B(\cH)$ is the
$*$-algebra of all bounded linear operators on $\cH$ equipped with the uniform norm $\left\|\cdot \right\|_\infty$, and
$\mathbf{1}$ is the identity operator on $\cH$.
Let $\mathcal{M}$ be
a von Neumann algebra on $\cH$.
We denote by $\cP(\cM)$ the collection  of all projections in $\cM$ and by $\cU(\cM)$ the collection of all unitary elements.
For details on von Neumann algebra
theory, the reader is referred to e.g.  \cite{KR1, KR2}
or \cite{Tak}. General facts concerning measurable operators may
be found in \cite{Nelson} and  \cite{Se} (see also
\cite[Chapter
IX]{Ta2} and the forthcoming book \cite{DPS}).
For convenience of the reader, some of the basic
definitions are recalled.

\subsection{$\tau$-measurable operators and generalized singular value functions}

A closed, densely defined operator $x:\mathfrak{D}\left( X\right) \rightarrow \cH $ with the domain $\mathfrak{D}\left( x\right) $ is said to be {\it affiliated} with $\mathcal{M}$
if $yx\subseteq xy $ for all $Y\in \mathcal{M}^{\prime }$, where $\mathcal{M}^{\prime }$ is the commutant of $\mathcal{M}$.
A  closed,
densely defined
operator $x:\mathfrak{D}\left( x\right) \rightarrow \cH $ affiliated with $\cM $ is said to be
{\it measurable}  if  there exists a
sequence $\left\{ p_n\right\}_{n=1}^{\infty}\subset \cP\left(\mathcal{M}\right)$, such
that $p_n\uparrow \mathbf{1}$, $p_n(\cH)\subseteq\mathfrak{D}\left(X\right) $
and $\mathbf{1}-p_n$ is a finite projection (with respect to $\mathcal{M}$)
for all $n$.
 The collection of all measurable
operators with respect to $\mathcal{M}$ is denoted by $S\left(
\mathcal{M} \right) $, which is a unital $\ast $-algebra
with respect to strong sums and products (denoted simply by $x+y$ and $xy$ for all $x,y\in S\left( \mathcal{M%
}\right) $).

Let $X$ be a self-adjoint operator affiliated with $\mathcal{M}$.
We denote its spectral measure by $\{E^X\}$.
 It is well known that if
$X$ is an operator affiliated with $\mathcal{M}$ with the
polar decomposition $X = U|X|$, then $U\in\mathcal{M}$ and $E\in
\mathcal{M}$ for all projections $E\in \{E^{|X|}\}$. Moreover,
$X\in S(\mathcal{M})$ if and only if  $E^{|X|}(\lambda,
\infty)$ is a finite projection for some $\lambda> 0$. It follows
immediately that in the case when $\mathcal{M}$ is a von Neumann
algebra of type $III$ or a type $I$ factor, we have
$S(\mathcal{M})= \mathcal{M}$. For type $II$ von Neumann algebras,
this is no longer true. From now on, let $\mathcal{M}$ be a
semifinite von Neumann algebra equipped with a faithful normal
semifinite trace $\tau$.

An operator $x\in S\left( \mathcal{M}\right) $ is called $\tau$-measurable if there exists a sequence
$\left\{p_n\right\}_{n=1}^{\infty}$ in $P\left(\mathcal{M}\right)$ such that
$p_n\uparrow \mathbf{1}$, $p_n\left(\cH\right)\subseteq \mathfrak{D}\left(x\right)$ and
$\tau(\mathbf{1}-p_n)<\infty $ for all $n$.
The collection $S\left( \mathcal{M}, \tau\right)
$ of all $\tau $-measurable
operators is a unital $\ast $-subalgebra of $S\left(
\mathcal{M}\right) $.
It is well known that a linear operator $x$ belongs to $S\left(
\mathcal{M}, \tau\right) $ if and only if $x\in S(\mathcal{M})$
and there exists $\lambda>0$ such that $\tau(E^{|x|}(\lambda,
\infty))<\infty$.
Alternatively, an unbounded operator $x$
affiliated with $\mathcal{M}$ is  $\tau$-measurable~(see
\cite{FK}) if and only if
$$\tau\left(E^{|x|}\bigl(n,\infty\bigr)\right)\rightarrow 0,\quad n\to\infty.$$
The set of all self-adjoint elements in $S(\mathcal{M},\tau)$ is denoted by $S(\mathcal{M},\tau)_h$,
which is a real linear subspace of $S(\mathcal{M},\tau)$.
The set of all positive elements in $S(\mathcal{M},\tau)_h$ is denoted by $S(\mathcal{M},\tau)_+$.

\begin{definition}\label{mu}
Let a semifinite von Neumann  algebra $\mathcal M$ be equipped
with a faithful normal semi-finite trace $\tau$ and let $x\in
S(\mathcal{M},\tau)$. The generalized singular value function $\mu(x):t\rightarrow \mu(t;x)$ of
the operator $x$ is defined by setting
$$
\mu(s;x)
=
\inf\{\left\|xp\right\|_\infty:\ p\in \cP(\cM)\mbox{ with}\ \tau(\mathbf{1}-p)\leq s\}.
$$
\end{definition}
An equivalent definition in terms of the
distribution function of the operator $x$ is the following. For every self-adjoint
operator $x\in S(\mathcal{M},\tau),$ setting the \emph{spectral distribution function} of $x $ by
$$d_x(t)=\tau(E^{x}(t,\infty)),\quad t>0.$$
It is clear that the function $d(x):\mathbb{R}\rightarrow[0,\tau(\mathbf 1)]$ is decreasing and the normality of the trace implies that $d(x)$ is right-continuous.
We have (see e.g. \cite{FK})
\begin{align}\label{dis}
\mu(t; x)=\inf\{s\geq0:\ d_{|x|}(s)\leq t\}.\end{align}
It is obvious  \cite[Remark 3.3]{FK} that \begin{align}\label{disx}
d(|x|)=d(\mu (x)).
\end{align}
An element $x\in S(\cM,\tau)$ is said to be $\tau$-compact if $\mu(t;x)\to 0$ as $t\to \infty$.
We denote by
$S_0(\cM,\tau)$ the subspace of $S(\cM,\tau)$ consisting of all $\tau$-compact elements in $S(\cM,\tau)$.


Consider the algebra $\mathcal{M}=L^\infty(0,\infty)$ of all
Lebesgue measurable essentially bounded functions on $(0,\infty)$.
The algebra $\mathcal{M}$ can be seen as an abelian von Neumann
algebra acting via multiplication on the Hilbert space
$\mathcal{H}=L^2(0,\infty)$, with the trace given by integration
with respect to Lebesgue measure $m.$
It is easy to see that the
algebra of all $\tau$-measurable operators
affiliated with $\mathcal{M}$ can be identified with
the subalgebra $S(0,\infty)$ of the algebra of Lebesgue measurable functions $L_0(0,\infty)$ which consists of all functions $x$ such that
$m(\{|x|>s\})$ is finite for some $s>0$.
It should also be pointed out that the
generalized singular value function $\mu(x)$ is precisely the
decreasing rearrangement $\mu(x)$ of the function $|x|$ (see e.g. \cite{KPS}) defined by
$$\mu(t;x)=\inf\{s\geq0:\ m(\{|x|\geq s\})\leq t\}.$$
%
%

If $\mathcal{M}=B(\cH)$ (respectively, $l_\infty$) and $\tau$ is the
standard trace ${\rm Tr}$ (respectively, the counting measure on
$\mathbb{N}$), then it is not difficult to see that
$S(\mathcal{M})=S(\mathcal{M},\tau)=\mathcal{M}.$ In this case,
for $x\in S(\mathcal{M},\tau)$ we have
$$\mu(n;x)=\mu(t; x),\quad t\in[n,n+1),\quad  n\geq0.$$
The sequence $\{\mu(n;x)\}_{n\geq0}$ is just the sequence of singular values of the operator~$x$.

\subsection{Classic Hardy--Littlewood--P\'{o}lya majorization and submajorization}
Let $(L_1(0,\infty),\left\|\cdot\right\|_{L_1(0 ,\infty)})$ be the $L_1$-space of Lebesgue measurable functions on the semi-axis $(0,\infty)$.
The pair
$$L_1(\cM,\tau)=\{x\in S(\cM,\tau):\mu(x)\in L_1(0,\infty)\},\quad \left\|x\right\|_{L_1(\cM,\tau)}:=\left\|\mu(x)\right\|_{L_1(0,\infty)}$$
defines a Banach bimodule affiliated with $\cM$
\cite{Kalton_S} (see also \cite{LSZ,DP2}).
For brevity, we denote $\left\|\cdot \right\|_{L_1(\cM,\tau)}$ by $\left\|\cdot \right\|_1$.
Clearly, $L_1(\cM,\tau)\subset S_0(\cM,\tau)$.
Using the extended trace $\tau:S(\mathcal{M},\tau)_+ \rightarrow[0, \infty]$ to a linear functional on $S(\mathcal{M},\tau),$ denoted again by $\tau$, the noncommutative $L_1$-space 
can be defined by
$$L_1(\mathcal{M},\tau)=\{x\in S(\mathcal{M},\tau):\tau(|x|)<\infty\}.$$
(see e.g. \cite{DDP93,LSZ}).
Let us denote
$L_1(\mathcal{M},\tau)_h=\{x\in L_1(\mathcal{M},\tau):x=x^{\ast}\}$ and $L_1(\mathcal{M},\tau)_+=L_1(\mathcal{M},\tau)\cap S(\mathcal{M},\tau)_+ .$

If $x,y\in S(\cM,\tau)$, then $x$ is said to be submajorized (in the sense of Hardy--Littlewood--P\'{o}lya) by $y$, denoted by $x\prec\prec y$, if \begin{align*} \int_{0}^{t} \mu(s;x) ds \le \int_{0}^{t} \mu(s;y) ds \end{align*} for all $t\ge 0$ \cite{LSZ}.
In particular, for $x,y\in S(0,\infty)$,   $x\prec \prec y$ if and only if  $\int_{0}^{t} \mu(s;x) ds \le \int_{0}^{t} \mu(s;y) ds $, $t\ge 0$.
For any $x,y\in S(\cM,\tau)$, we have \cite{FK,DDSZ}
\begin{align}\label{2.4}
  \mu(xy)\prec \prec  \mu(x)\mu(y).
\end{align}

Assume that  $\mathcal{M}$ is a finite  von Neumann algebra equipped with a faithful normal finite tracial state $\tau$.
We note that the space $S(\cM,\tau)$ is the set of all densely defined closed linear operators $x$ affiliated
with $\mathcal{M} $.
However, if the trace $\tau$ is infinite, then there are densely defined closed linear operators which are not $\tau$-measurable.


We introduce the notion of spectral scales (see  \cite{P}, see also \cite{HN1, HN2, DDP, DPS,AM}).
If $x\in S(\mathcal{M},\tau)_h$, then the \emph{spectral scales} (also called \emph{eigenvalue functions}) $\lambda(x):[0,1)\rightarrow(-\infty,\infty]$ 
$$\lambda(t;x)=\inf\{s\in\mathbb{R}\ :\ d(s;x)\leqslant t\},\ \ t\in[0,1).$$
The spectral scale $\lambda(x)$   is decreasing right-continuous functions.
If $x\in S(\mathcal{M},\tau)_+$, then it is evident that $\lambda(t;x)=\mu(t;x)$ for all $t\in[0,1)$.

Assume that  $\mathcal{M}=L_{\infty}(0,1)$ and $\tau(f)=\int_{0}^1 fdm$,
 where  $m$ the   Lebesgue measure on $(0,1)$.
 In this case, $S (\mathcal{M},\tau)_h$ consists of all real measurable functions $f$ on $(0,1)$. For every $f$, $\lambda(f)$ coincides with  the \emph{right-continuous equimeasurable nonincreasing rearrangement} $\delta_f$ of $f$ (see e.g. \cite{HN1}):
$$\lambda(t;f)=\delta_f(t)=\inf\{s\in\mathbb{R}:\ m(\{x\in X:\,f(x)>s\})\leqslant t\},\ \ t\in[0,1 ).$$

We note that for every $x\in L_1(\cM,\tau)_h$, we have (see e.g. \cite{DDP}, \cite[Proposition 1]{P} and \cite[Chapter III, Proposition 5.5]{DPS})
\begin{equation}\label{spectral scale}
  \tau(x)=\int_0^1\lambda(t;x)dt.
\end{equation}




For every $f,g\in L_1(\cM,\tau),$ $g$ is said to be  \emph{majorized} by $f$ in the sense of Hardy--Littlewood--P\'{o}lya (written by $g\prec f$)  if
$$\int_0^s\lambda(t;g)dt\leqslant\int_0^s\lambda(t;f)dt$$
for all $s\in[0,1)$ and
$$\int_0^{1}\lambda(t;g)dt=\int_0^{1}\lambda(t;f)dt.$$

For  a self-adjoint element $y\in L_1(\mathcal{M},\tau)_h$, we denote
$$\Omega(y):=\{x\in L_1(\mathcal{M},\tau)_h\ :\ x\prec y\}.$$
We note \cite[Theorem 3.2]{HN2}  that
\begin{align}\label{tri}
\lambda (x) -\lambda (y)\prec \lambda (x- y), ~\forall x,y\in L_1(\cM,\tau)_h.
\end{align}

Let $\cM$ be a semifinite von Neumann algebra.
For every $x\in S_0(\cM,\tau)_+$ and  $s\in \mathbb{R}$, if $e\in \cP(\cM)$ is such that $$E^x(s,\infty) \le e \le E^x [ s,\infty ),$$
then  (see \cite{DDP}  or \cite[Chapter III, Proposition  2.10 and Lemma 7.10]{DPS}):
\begin{align}\label{2.2}
\mu (t;xe) =\mu (t;x) \mbox{ for all } t\in [0,\tau(e)),
\end{align}
and
\begin{align}\label{2.3}
\mu (t;xe^\perp ) =\mu (t+\tau(e);x) \mbox{ for all } t\in [0,\tau(e^\perp )).
\end{align}


\subsection{Hardy--Littlewood--P\'{o}lya majorization in the infinite setting}
The definition of classic  Hardy--Littlewood--P\'{o}lya majorization can not be extended to the infinite setting directly.
We adopt the definition suggested by Hiai and Nakamura \cite{HN1} (see also \cite{NRS,ST}).
 For definition in the setting of positive infinite sequences, see \cite{Komiya,KW,LW15,Weiss,KW11,JLW,KW08}).
\begin{definition}\label{definfi}
Let $\cM$ be a semifinite von Neumann algebra equipped with a semifinite faithful normal trace $\tau$.
Let $x,y\in L_1(\cM,\tau)_h$.
We say that $x$ is majorized by $y$ (in the sense of Hardy--Littlewood--P\'{o}lya, denoted by $x\prec y$) if
$x_+\prec \prec y_+$, $x_-\prec \prec y_-$ and $\tau(x)=\tau(y)$.
\end{definition}
When $\tau$ is finite,  Definition \ref{definfi} coincides with the classic definition of Hardy--Littlewood--P\'{o}lya majorization \cite[Proposition 1.3]{HN1}.

We show the equivlance between (b) and (c) in Theorem \ref{main} (similarly results for positive operators can be found in  \cite{NRS,Skoufranis,Hiai}).

  \begin{proposition}\label{treq}Let $\cM$ be a semifinite von Neumann algebra equipped with a semifinite faithful normal trace $\tau$.
Let $x,y\in L_1(\cM,\tau)_h$.
Then,
  $x\prec y$ if and only if
  $\tau((x-t)_+) \le \tau((y-t)_+)$ and  $\tau((-x-t)_+) \le \tau((-y-t)_+)$  for all $t\in \bR$ (or $\bR_+$)  and $\tau(x)=\tau(y)$.
  \end{proposition}
  \begin{proof}
  By \cite[Proposition 2.3]{Hiai} (see also \cite[Theorem 4]{SZ11}),
  $x_+\prec \prec y_+$ implies that $\tau((x-t)_+) =\tau((x_+-t)_+) \le \tau((y_+-t)_+)=\tau((y_+-t)_+)$ for all $t>0$.
  The same argument shows that $\tau((x_--t)_+) \le \tau((y_--t)_+)$ for all $t>0$.
  When $t \le 0$, by the spectral theorem,
  we have $$\tau((x-t)_+) =\infty =\tau((y -t)_+)$$
  and $$\tau((-x-t)_-)  =\infty =  \tau((-y -t)_-),$$
  which prove the ``$\Rightarrow$'' implication.

  The ``$\Leftarrow$'' implication follows immediately from  \cite[Proposition 2.3]{Hiai}.
  \end{proof}

\begin{rmk}If $\tau$ is finite, then
the sufficient condition in Proposition \ref{treq} is equivalent with
$\tau((x-t)_+) \le \tau((y-t)_+)$  for all $t\in \bR$  and $\tau(x)=\tau(y)$ (see \cite[Proposition 1.2. (1) and Proposition 1.3]{HN1}).
\end{rmk}

\begin{proposition}\label{HN1}\cite[Proposition 1.1]{HN1} (see also \cite{DPS,LSZ})
If $x\in L_1(\cM,\tau)_h$, then for every $0<s<\tau({\bf 1})$,
$$\int_0^s \mu   (t; x_+) dt =\sup \left\{\tau(xa): a\in \cM, 0\le a\le {\bf 1}, ~\tau(a)=s  \right\}.$$
If $\cM$ is non-atomic, then the above equality holds under the condition that the element $a$ varies in the set of projections $e\in \cP (\cM)$ with $\tau(e)=s$.
\end{proposition}
We draw reader's attention that \cite[Lemma 4.1]{FK}
contains an  inaccuracy. Namely,
the second assertion in  \cite[Lemma 4.1]{FK} is false in general
 (see \cite[Chapter III, Remark 9.7]{DPS}).

\begin{rmk}\label{alwayspositive}
For any $x\in L_1(\cM,\tau)_h$, $\lambda(x)$ is defined by $\lambda (s;x) =\inf\{t\in \bR : \tau(e_{(t,\infty)}(x))\le s\}$.
We note that whenever $\tau({\bf 1})=\infty$, $\lambda (x)\ge 0$ for any $x\in L_1(\cM,\tau)_h$.
Otherwise, $0> \lambda (s;x) =\inf\{t\in \bR : \tau(e_{(t,\infty)}(x))\le s\}$ for some $s>0$.
Hence, there exists $s'<0$ such that $\tau(e_{(s',\infty)}(x))\le s$, i.e.,
$\tau(e_{(-\infty ,s']}(x))=\infty $.
 This implies that $x$ is not $\tau$-compact,
 which contradicts  with the assumption that $x\in L_1(\cM,\tau)$.
 In particular,
by the  definition of eigenvalue functions, we obtain that $\lambda (x)=\lambda(x_+)=\mu(x_+)$ for any $x\in L_1(\cM,\tau)_h$ \cite[p.5]{HN1}.
\end{rmk}
\begin{proposition}\label{2.5}If  $\tau({\bf 1})=\infty$.
Let $a,b\in L_1(\cM,\tau)_h $.
We have
$$ \int_0^t \lambda (s;(a+b)_+  )ds   \le \int_0^t \lambda (s;a_+ )ds +\int_0^t \lambda (s;b_+  )ds  $$
and $$ \int_0^t \lambda (s;(a+b)_- )ds   \le \int_0^t \lambda (s;a_- )ds +\int_0^t \lambda (s;b_-  )ds  . $$
In particular, if $a,b\prec c \in L_1(\cM,\tau)_h $, then
$\frac{a+b}{2}\prec c $.
\end{proposition}
\begin{proof}
Without loss of generality, we may assume that $\cM$ is non-atomic (see e.g.~\cite[Lemma 2.3.18]{LSZ}).

Let $t>0$ be fixed.
For any $e\in \cP(\cM)$ such that $\tau(e)=t$,
by Proposition \ref{HN1},
we have
$$\tau((a+b)e) =\tau(a e )+\tau(b e )\le \int_0^t \lambda (s;a )ds +\int_0^t \lambda (s;b )ds. $$
By Proposition \ref{HN1}, we have
$$\int_0^t \lambda (s;a+b  )ds  \le \int_0^t \lambda (s;a )ds +\int_0^t \lambda (s;b )ds. $$
By  Remark \ref{alwayspositive}, we complete the proof for the first inequality.
The second inequality follows by taking $-a$ and $-b$.
The last assertion is a straightforward consequence of the above results.
\end{proof}

\section{Alberti--Uhlmann problem in the  setting of finite von Neumann algebras}\label{AUF}
We recall that definition of doubly stochastic operators on   von Neumann algebras, which were  introduced  by
Tregub \cite{Tregub81,Tregub86}, Kamei\cite{Kamei} and Hiai \cite{Hiai}.

\begin{definition}\cite{Hiai}
Let $\cA$ and $\cB$ be  semifinite von Neumann algebras equipped with semifinite faithful normal traces $\tau_1$ and $\tau_2$.
A positive linear map $\varphi:\cA\to \cB$ is said to be doubly stochastic if $\varphi({\bf 1}_\cA)={\bf 1}_\cB$ and $\tau_2\circ \varphi =\tau_1$ on $\cA_+$.
\end{definition}

Throughout this section, we always assume  that $\cM$ is a finite von Neumann algebra equipped with a faithful normal tracial state $\tau$.
We denote by $DS(\cM)$ the set of all doubly stochastic operators on $\cM$.
 Using a recent advance in \cite{DHS}, we provide  a complete  resolution of  Alberti--Uhlmann problem in this setting.

Note that  $\varphi$ is $\left\|\cdot \right\|_1$-bounded
on $L_1(\cM,\tau)\cap \cM$.
Hence, it can be extended to a bounded linear map on $L_1(\cM,\tau)$ (denoted by the same $\varphi$) \cite[Section 4]{Hiai}.
When the trace $\tau$ is finite, every doubly  stochastic operator $\varphi:\cM\to \cM$ is normal \cite[Remark 4.2 (2)]{Hiai}, i.e., $\varphi(x_i)\uparrow \varphi(x)$ if $x_i\uparrow x\in \cM_+$.
For any $y\in L_1(\cM,\tau)_h$,
we define
$$\Omega(y):=\{x\in L_1(\cM,\tau)_h: x\prec y \}. $$

The following result is folklore.
Due to the lack of suitable references, we provide a short proof below.
\begin{proposition}\label{propcc} Let $y\in L_1(\cM,\tau)_h$.
Then, $\Omega(y)$ is a convex set which is  closed in $L_1(\cM,\tau)$.
\end{proposition}
\begin{proof}
By \eqref{2.5}, $\Omega(y)$ is convex.
Let $\{x_n\}$ be a sequence in $\Omega (y)$ converging in $L_1$.
We define $x:= \left\|\cdot \right\|_1-\lim_{n\to \infty}x_n$.
Hence, for any $t\in (0,1)$, we have
\begin{align}\label{xxn}
\int_0^t  \lambda (s;x ) -\lambda (s;x_n  )  ds
&\stackrel{Prop. \ref{HN1}}{\le} \int_0^t \lambda\big(s; \lambda (x ) -\lambda (x_n  )  \big) ds
\stackrel{\eqref{tri}}{\le} \int_0^t \lambda  (s; x-x_n) ds  \nonumber \\
&\le \int_0^1 \lambda  \left(s; (x-x_n)_+ \right) ds  +\int_0^1 \lambda  \left(s; (x-x_n)_- \right) ds \nonumber\\
& =
\left\|x-x_n\right\|_1\to 0.
\end{align}
Since $x_n \prec y$, it follows that for any $t\in(0,1)$,
 $$ \int_0^t \lambda (s;y)ds \ge \int_0^t \lambda (s;x_n  ) ds \ge \int_0^t  \lambda (s; x )ds - \left\|x -x_n\right\|_1 \to \int_0^t  \lambda (s;x)ds$$
as $n\to \infty$.
Hence,  $ \int_0^t \lambda (s;y) \ge  \int_0^t  \lambda (s;x)ds$, $t\in (0,1)$.
The same argument with that of \eqref{xxn} infers that
$$
\int_0^1 \lambda (s;x_n) -\lambda (s; x  )  ds   \le
\left\|x-x_n\right\|_1\to 0.
$$
Hence, $\tau(y)=\tau(x_n)\le \tau(x)$, which together with $ \int_0^t \lambda (s;y)ds \ge  \int_0^t  \lambda (s;x)ds$
implies that $x\prec y$.
\end{proof}

\begin{rmk}

 Note that  $\Omega(y)$
 is $\left\|\cdot\right\|_1$-bounded.
  Moreover, for each $\varepsilon>0$,
 there exists a $\delta>0$ (see e.g. \cite[Theorem 3.1]{DPS16}, \cite[Lemma 6.1]{DP12}, \cite{HLS}) such that
 if $e\in \cP(\cM)$ with $\tau(e)<\delta$,
 then for all $x\in \Omega (y)$, we have
 \begin{align*}|\tau( xe )|=|\tau(x_+e)|+|\tau(x_- e)|
 &\stackrel{Prop. \ref{HN1}}{\le} \int_0^\delta \lambda (s;x_+)ds +\int_0^\delta \lambda (s;x_-)ds\\
 &\quad  \le  \int_0^\delta  \lambda (s;y_+) +\lambda (s;y_-)ds <\varepsilon,
 \end{align*}
 which shows that $\Omega(y)$ is relatively weakly compact (see e.g.
 \cite[Corollary 4.5]{DPS16},  see also \cite{Umegaki,RX,Akemann}).
 Since $\Omega(y)$ is a $\left\|\cdot\right\|_1$-closed
convex subset in $L_1(\cM,\tau)$, it follows that $\Omega(y)$ is weakly closed  \cite[Chapter V., Theorem 1.4]{Conway}.
Then, by the Krein--Milman theorem, $\Omega(y)$ is $\sigma(L_1,L_\infty)$-closure of the convex hull of all extreme points on $\Omega(y)$.
Again, by  \cite[Chapter V., Theorem 1.4]{Conway},
$\Omega(y)$ is $L_1$-closure of the convex hull of all extreme points on $\Omega(y)$.
\end{rmk}


The assumption that the operators are positive  plays no role in the proof of \cite[Theorem 4.7. (1)]{Hiai}.
So, the ``only if'' part below follows from the same argument.
Using a result by Day \cite{Day},
we provide below a similar (but simpler)  proof  for completeness.
We note that the special case for finite factors was described in  \cite{HN1} (see also \cite[Theorem 2.18]{Skoufranis}).
The following  extends \cite{Kamei} (see also \cite[Theorem 2.18]{Skoufranis} and \cite[Theorem 4.7. (1)]{Hiai}).
\begin{theorem}\label{sto}
Let $x,y\in L_1(\cM,\tau)$, then $x\prec y $ if and only if there exists a $\varphi \in DS(\cM)$ such that $x=\varphi(y)$.
\end{theorem}
\begin{proof}
$\Rightarrow$.
Let $\cA$ (resp. $\cB$) be the commutative von Neumann subalgebra of $\cM$ generated by all spectral projections of $y$ (resp. $x$).
Since $(\cA,\tau)$ and $(\cB,\tau)$ are two normalized measure spaces, by \cite[Theorem 4.9]{Day}, there exists  a doubly stochastic operator $\psi: \cA\to \cB  $ such that $ x=\psi(y )$.
Taking the conditional expectation $\cE :\cM \to \cA$ \cite[Proposition 2.1]{DDPS},
we define $\varphi:\cM\to \cM$ by
$\varphi(z) =   \psi \circ \cE $, $z\in \cM$.
 Clearly, $\varphi$ is a doubly stochastic operator and $\varphi(y)=x$.

$\Leftarrow$.
Assume that there exists  a $\varphi \in DS(\cM)$ such that $x=\varphi(y)$.
By the definition of doubly stochastic operators, we have $\tau(x)=\tau(y)$.
Let $$y_n : = ye^{y}(-n,\infty)   +  \frac{1}{1- \tau(e^{y}(-n,\infty))} \int_{\tau(e^{y}(-n,\infty))}^{1} \lambda (s;y)ds .$$
In particular, $y_n\prec y$ (see e.g. \cite[p.198]{SZ}).
By \cite[Theorem 4.5 (2)]{Hiai},
we have $\varphi(y_n+n )\prec y_n+n$, i.e., $\varphi(y_n )\prec y_n \prec y $.
By \cite[Proposition 4.1]{Hiai},
we have $$\left\|x- \varphi(y_n)\right\|_1=\left\|\varphi(y-y_n)\right\|_1\le \left\|y-y_n\right\|_1\to 0.$$
It follows from Proposition \ref{propcc} that $x\prec y$.
 \end{proof}

 Combining Theorem \ref{sto}, \cite[Proposition 1.3]{HN2} and \cite[Theorem 1.1]{DHS} (with Krein--Milman theorem), we  present the finite von Neumann algebra version of \cite[Theorem 2.2]{AU},
 which provides a complete answer  to
Alberti and Uhlmann's problem \cite[Chapter 3]{AU} in this setting.
 \begin{theorem}\label{th:f}
 For any $x,y\in L_1(\cM,\tau)_h$, the following conditions are equivalent:
 \begin{enumerate}
   \item[(a).]  $x$ is in the doubly stochastic orbit of $y$;
   \item[(b).]    $x\prec y$;
  \item[(c).]$x$ belongs to the $L_1$-closure of the  convex hull of all elements $x\in L_1(\cM,\tau)$ satisfying that for any $t\in  (0,1)$, one of the following options holds:
       \begin{enumerate}
  \item[(i).]   $\lambda(t;x)= \lambda(t;y)$;
  \item [(ii).]
   $\lambda(t;x) \ne \lambda(t;y)$ with  the spectral projection $ E^x \{\lambda (t;x)  \} $ being an atom in $\cM$ and $$\int_{ \{s;\lambda (s;x)=\lambda (t;x)\}}  \lambda (s;y)ds    =  \lambda(t ;x)     \tau(E^x (\{\lambda (t;x)\}))  .  $$
\end{enumerate}
  \item[(d).]$x$ belongs to the $L_1$-closure of the  convex hull of all elements $x\in L_1(\cA,\tau)$ satisfying that for any $t\in  (0,1)$, one of the following options holds:
       \begin{enumerate}
  \item[(i).]   $\lambda(t;x)=  \lambda(t;y)$;
  \item [(ii).]
   $\lambda(t;x) \ne \lambda(t;y)$ with  the spectral projection $ E^x \{\lambda (t;x)  \} $ being an atom in $\cA$ and $$\int_{ \{s;\lambda (s;x)=\lambda (t;x)\}}  \lambda (s;y)ds    =  \lambda(t ;x)     \tau(E^x (\{\lambda (t;x)\}))  .  $$
\end{enumerate}
Here, $\cA$ is an abelian von Neumann subalgebra of $\cM$ such that $x\in L_1(\cA,\tau)$\footnote{
Since $x\prec y$ is equivalent with $x\prec \lambda (y)$ and \cite[Theorem 1.1]{DHS} holds for $x\in L_1(\cA,\tau)$ and $\lambda (y)\in L_1(0,1)$ (see the the beginning of \cite[Section 3]{DHS}), it follows    that the extreme points of $\{z\prec \lambda(y):z\in L_1(\cA,\tau)\}$ are those satisfying conditions (i) and (ii).}.
   \item[(e).] $\tau(x) =\tau(y)$ and $\tau((x-r)_+)\le \tau((y-r)_+)$ for all $r\in \bR$;
   \item[(f).] $\tau(f(x)) \le \tau(f(y))$ for all convex function $f$ on $\bR$;
   \item[(g).] $\tau(|x-r|)\le \tau(|y-r|)$ for all $r\in \bR$;
   \item[(h).]$ \int_0^t\lambda (s;f(x))ds \le \int_0^t \lambda (s; f(y))ds$ for all convex function $f$ on $\bR$.
 \end{enumerate}
 If, in addition, $\cM$ is a factor, then    $\overline{co \{uyu^*:u\in \cU(\cM)\}}^{\|\cdot\|_1}$ coincides with the doubly stochastic orbit of $y$\cite[Theorem 2.3]{HN1}.
 \end{theorem}

\section{On Hiai's conjecture: Ryff's problem in the infinite setting}\label{s:h}

Throughout this section, we always assume that $L_\infty (0,\infty)$ is the space of all (classes of) bounded real  Lebesgue measurable function on $(0,\infty)$.
  Let $x,y\in L_1(0,\infty)$.
 Recall that   $x$ is said to be majorized by $y$ (denoted by $x\prec y$) in the sense of Hardy--Littlewood--P\'{o}lya  if
  $x_+\prec \prec y_+$, $x_-\prec \prec y_-$ and $\int_{\bR_+}x dt=\int_{\bR_+}y dt $.

Following Birkhoff's result \cite{Birkhoff46} and his problem 111 \cite{Birkhoff48},
 numerous results on  doubly stochastic operators in the infinite setting appeared  (see e.g. \cite{Kendall,Isbell,Isbell55,Revesz,RP,Komiya,ST,Hiai,Zanin, Safarov05,Safarov06,PR} and references therein).
Hiai's proof of Theorem 4.7 in \cite{Hiai} relies on the so-called BW-compactness \cite{Arveson} of $DS(\cM)$.
We present several examples below showing that Hiai's conjecture (stated in the introduction) is true in the setting of $\sigma$-finite infinite von Neumann alagebras.

Recall that any doubly stochastic operator on $\cM$ can be canonically extended to a linear map on $L_1(\cM,\tau)+\cM$ \cite[Section 4]{Hiai}.
Hence, by a doubly stochastic operator $\varphi$ on $\cM$, we always mean that $\varphi$ acts on $L_1(\cM,\tau)+\cM$.

The first example concerning on  normal  doubly stochastic operators (for a doubly stochastic operator $\varphi$ on a von Neumann algebra $\cM$, $\varphi$ is said to be normal if $\varphi(x_i)\uparrow \varphi(x)$ if $x_i\uparrow x\in \cM_+$).
 Hiai\cite{Hiai} showed that a doubly stochastic operator on a finite von Neuman algebra is necessarily normal and he gave an example of doubly stochastic operator on an infinite von Neumann algebra, which is not normal.
We show that one cannot expect that for any $x\prec y \in L_1(0,\infty)$, there exists a normal doubly stochastic operator such that $\varphi(y)=x$.
\begin{example}
There exists $x\prec y \in L_1(0,\infty)$ such that there exists no normal doubly stochastic operator $\varphi$ such that $\varphi(y)=x$.
Let $y=\mu(y)$ be an arbitrary strictly positive integrable function on $(0,\infty)$.
We define $x$ by $x(t) =y(t-1)$, $t\ge 1$, and $x(t)=0$, $0\le t<1$.
Clearly, $x\prec y$.
Assume that there exists a  normal doubly stochastic operator $\varphi:L_\infty (0,\infty)\to L_\infty (0,\infty )$ such that (the extension of $\varphi$ on $L_1(0,\infty)$) $\varphi(y)=x$.
Then, for any indicator function $ \chi_{(0, n]}$, $n\ge 0$, we have
$$\mu(n;y)\varphi(\chi_{(0, n]}) \le  \varphi(y). $$
In particular,
$\varphi(\chi_{(0, n]})= \varphi(\chi_{(0, n]}) \chi_{[1, \infty ) } $.
Note that
$\varphi(\chi_{(0, n]})\le \varphi(\chi_{\bR_+}) \le \chi_{\bR_+}$.
We obtain that
$$ \varphi(\chi_{(0, n]}) \le  \chi_{[1, \infty ) }  . $$
Since $\chi_{(0,n]}\uparrow \chi_{\bR_+}$, it follows that
$\chi_{\bR_+}=\varphi(\chi_{\bR_+})=\sup \varphi(\chi_{(0, n]}) \le  \chi_{[1, \infty ) }  $,
which is a contradiction.
\end{example}
The second example shows that
if one consider $L_1(\cM,\tau)+\cM$ (or simply $\cM$), then there exists $x\prec y$ such that there exists no doubly stochastic operator $\varphi$ such that $\varphi(y)=x$.
\begin{example}
Let $y:=\chi_{\bR_+}$ and $x:=\chi_{(1,\infty)}$.
Clearly, $x\prec y$.
Assume that $x=\varphi(y)$ for some doubly stochastic operator $\varphi$.
Then, $\chi_{\bR_+}= \varphi(y) = x =\chi_{(1,\infty)}$,
which is a contradiction.
\end{example}

The following theorem shows that even if we restrict $x,y$ in $L_1(0,\infty )$ and relax the restriction (i.e. normality) on $\varphi$, there  are still element $x$ with $x\prec y $,  which are not in the doubly stochastic orbit of $y$, which again confirms Hiai's conjecture.
In Lemma~\ref{atomnotequ}, we will construct a similar but much more complicated example in a more general setting.
\begin{example}\label{th:counter}
Let $\cM=\l_\infty$ equipped with the standard trace.
There exist $x\prec y\in \ell_1 $ such that there exists no doubly stochastic operator $\varphi:l_\infty \to l_\infty $ with $\varphi(y)=x$.

Let $y: =(0,1,\frac{1}{2},\cdots,\frac{1}{2^n},\cdots)$ and $x:=(1,\frac{1}{2},\cdots,\frac{1}{2^n},\cdots)$.
Assume that there exists a doubly stochastic operator $\varphi$ on $l_\infty$ such that $\varphi(y)=x$.
Since every doubly stochastic  operator is bounded on $\ell_1 $ (with respect to the $L_1$-norm), it follows that
the restriction of $\varphi$ on $\ell_1$ can be represented as an infinite matrix $(a_{ij})$ (see e.g. \cite[Example 4.13]{deAK}).
Since $\varphi$ preserves the trace, it follows that $\sum _{i=1}^\infty a_{ij} =1$ for every $j$.
Moreover, since $\varphi( \sum_{i=1}^n e_n) \le \varphi(\bf 1)={\bf 1}$ and $\varphi$ is positive, it follows that $\sum_{j=1}^n a_{ij} \le 1$ for every $n$ and $a_{ij}\ge 0$ for any $1\le i,j<\infty $.

By $\varphi (y)=x$, we obtain that $\sum_{j=1}^\infty  a_{ij}y_{j} = x_i $, i.e., $\sum_{j=2}^\infty  a_{ij}\frac{1}{2^{j-2}} = \frac{1}{2^{i-1}  }$.
When $i=1$,
we have
$$\sum_{j=2}^\infty  a_{1j}\frac{1}{2^{j-2 }} =1. $$
Since  $\sum_{j=1}^n a_{1j} \le 1$, it follows that $a_{12}=1$ and $a_{1j}=0$ whenever $j\ne 2$.
Recall that
$\sum_{i=1}^\infty a_{i2}=1$.
We obtain that $a_{i2}=0$ whenever $i\ne 1$.
Arguing inductively, we show that
$$a_{n,n+1}=1, ~n\ge 1$$
and $$a_{i,j}=0, ~i+1\ne j.$$
In particular,
$a_{i1}=0$ for any $i\ge 1$.
Hence,
$\varphi(e_1) =0$,
which contradicts with the assumption that $\varphi$ preserves the trace.
\end{example}
\begin{rmk}
It is shown in \cite[Corollary 6.1,  S(ii $'$)]{KW} that if $x=\mu(x),y=\mu(y)\in \ell_1$ such that $x\prec  y$, then there exists a doubly stochastic operator $\varphi$ such that $x= \varphi(y)$ (see also \cite[Corollary 2]{Komiya} and \cite[Theorems 4 and 5]{ST} for similar results).
\end{rmk}

We show below that for any $\sigma$-finite infinite von Neumann algebra, there always exists $y$ such that the doubly stochastic orbit of $y$ does not  coincide with the orbit of $y$ in the sense of Hardy--Littlewodd--P\'{o}lya.
Before proceeding, we need  the following generalization of Example \ref{th:counter}.
The key instrument of the proof below is the description of extreme points of $\Omega(y)$ in the setting when   atoms in  $l_\infty$ have  different traces  (see \cite{DHS}, see also Theorem \ref{th:ext}).
\begin{lemma} \label{atomnotequ}
Let $\cM=\l_\infty$ equipped with a semifinite faithful normal  trace $\tau$ such that $\tau(e_1)< \tau(e_2)< \cdots < \tau(e_n) <  \cdots $,
where $e_1,e_2,\cdots$ are the unit elements of $l_\infty$.
There exist $x\prec y\in \ell_1 $ such that there exists no doubly stochastic operator $\varphi:l_\infty \to l_\infty $ with $\varphi(y)=x$.
\end{lemma}
\begin{proof}
For the sake of convenience, we denote $w_n=\tau(e_n)$ for any $n\ge 1$.  

Let $$y: =\left(0,\frac{1}{w_2},\frac{1}{2w_3},\cdots,\frac{1}{2^{n-2} w_n},\cdots,\right) \in \ell_1$$
 and
 $$x:=\left(\frac{1}{w_2}, \frac{\frac{1}{w_2}(w_2-w_1) + \frac{w_1}{2w_3}  }{w_2},  \frac{ \frac{1 }{2w_3} (w_3-w_1)  +\frac{w_1}{4w_4}   }{w_3} , \cdots, \frac{ \frac{1 }{2^{n-3}w_{n-1}} (w_{n-1}-w_1)  +\frac{w_1}{2^{n-2}w_n}   }{w_{n-1}}  ,\cdots  \right).$$
 Note that
 $\mu(y)= \sum_{n\ge 2} \frac{1}{2^{n-2}w_n}\chi_{[\sum_{k=2}^{n-1}w_k  , \sum_{k=2}^{n} w_k)}$.
 Define an averaging operator~\cite{SZ,LSZ} $E: L_\infty (0,\infty)\to  L_\infty (0,\infty) $ by
 $$ E(f)= \sum_{n\ge 1}  \frac{1}{w_n} \int_{\sum_{k=1}^{n-1}w_k}^{\sum_{k=1}^{n} w_k} f(s)ds\cdot \chi_{[\sum_{k=1}^{n-1}w_k  , \sum_{k=1}^{n} w_k)}  .$$
 In particular,
 $E(\mu(y))=\mu(x)$.
 Recall that averaging operators are doubly stochastic operators (see e.g. \cite[Lemma 3.6.2]{LSZ}, \cite[p.198]{SZ} and \cite{Zanin}).
 We obtain that  $x\prec y$ (in particular, $x$ is an extreme point of
 $\Omega (y)$, see Theorem \ref{th:ext}).

Assume that there exists a doubly stochastic operator $\varphi$ on $l_\infty$ such that $\varphi(y)=x$.
For every  $n$, we denote
$$\varphi(e_n) = (a_{1,n},a_{2,n}, \cdots , a_{k,n},\cdots). $$
Since $\varphi$ is positive, it follows that $a_{kn}\ge 0$, $k,n\ge 1$.
Since $\varphi$ preserves the trace and $\tau(e_n)=w_n$,
it follows that
$$w_n =  \sum _{k=1}^\infty a_{k,n} w_k    .  $$
Moreover, since $\varphi({\bf 1})={\bf 1}$, it follows that
$$\sum_{n=1}^\infty  a_{k,n} \le 1 $$
for every $k  \ge 1$.

Since the restriction of $\varphi$ on $\ell_1$ is continuous in $\left\|\cdot \right\|_1$, it follows that $\varphi$ is normal on $\ell_1$ (see e.g. \cite[Proposition 2 (iii)]{DP2}).
We may view the restriction of $\varphi$ on $\ell_1$ as an infinite matrix $(a_{k,n})$.

Since $\sum_{n=1}^\infty a_{1,n}\le 1$ and
$$\sum_{n=1}^\infty a_{1,n} y_n =\sum_{n=2}^\infty a_{1,n} y_n =x_1 = y_2>y_n, ~n> 2,$$
it follows that $a_{1,2} =1$ and $a_{1,n}=0$ when $n\ne 2$.
Arguing inductively,
 we obtain that
\begin{equation*}
(a_{k,n}) =
\begin{pmatrix}
0 & 1  &  0 & 0 &\cdots &   \\
0 & \frac{w_2 -w_1}{w_2} & \frac{w_ 1}{w_2} & 0 &\cdots &   \\
0 & 0 &  \frac{w_3 -w_1}{w_3}&  \frac{w_1}{w_3} &\cdots &   \\
\vdots  &  \vdots  & \vdots &\vdots &\ddots
\end{pmatrix}
\end{equation*}
Hence,
$\varphi(e_1) =0$,
which contradicts with the assumption that $\varphi$ preserves the trace.
\end{proof}

Now, we are ready to show that Hiai's conjecture is true for any $\sigma$-finite infinite von Neumann algebra.
\begin{theorem}\label{sigmanot}
Let $\cM$ be a $\sigma$-finite  von Neumann algebra equipped with a semifinite infinite faithful normal trace.
There are $x,y\in L_1(\cM,\tau)$ such that there exists no doubly stochastic operator $\varphi$ on $\cM$ such that $\varphi(y)=x$.
\end{theorem}
\begin{proof}
Since $\cM$ is $\sigma$-finite, it follows that there exists a sequence $\{p_n\}_{n\ge 1}$  of disjoint $\tau$-finite projections such that $\vee p_n ={\bf 1}$.
Moreover, we may assume that $\tau(p_n) $ is strictly increasing.
The $\left\|\cdot\right\|_1$-closure of the linear span of $\{p_n\}$ can be viewed as the $\ell_1$-sequence space generated by atoms $\{p_n\}$.
By Lemma \ref{atomnotequ}, there exists $x,y\in \ell_1$ such that there exists no doubly stochastic operator $\varphi:\ell_1\to \ell_1$ with $\varphi(y)=x$.

Assume that there exists a doubly stochastic operator $\varphi:\cM\to \cM$ such that $\varphi(y)=x$.
Let $\cE$ be the conditional expectation from $\cM$ to $\ell_\infty$ (the von Neumann subalgebra of $\cM$ generated by $\{p_n\}$).
Clearly, $\cE\circ \varphi :\ell_\infty \to \ell_\infty$ is a doubly stochastic operator which satisfies that $\cE\circ \varphi  (y) =\cE (x)= x$.
This contradicts with the assumption imposed on  $x$ and $y$.
\end{proof}

\section{A decomposition theorem}

In view of Theorem \ref{sigmanot}, to study Alberti--Uhlmann problem in the infinite setting, we have to relax the restriction on the doubly stochastic operators.
In the next section, we will show that there exists a normal doubly stochastic operator $\varphi$ on a larger algebra such that $\varphi(y)=x$ provided $x\prec y$.
Before proceeding to this result,  we prove a decomposition theorem for Hardy--Littlewood-P\'{o}lay majorization in this section.


Recall a  result due to  Ryff that \cite[Chapter 2, Corollary 7.6]{BS} (see also \cite{R}) for any $0\le f\in L_1(0,\infty)$, there exists a measure-preserving transformation  $\sigma $ (i.e., $m(\sigma^{-1}(E)) = m(E)$ for any measurable set $E$) from the support $s(f)$ of $f$ onto the support of $\mu(f)$ such that
$$f = \mu(f) \circ \sigma .$$
If $f\in L_1(X,\nu)_h$ ($(X,\nu)$ is a finite non-atomic measure space), then
there exists a measure-preserving transformation  $\sigma $   from   $X$ onto $(0,\nu(X))$
 such that \cite{R}
$$f = \lambda (f) \circ \sigma .$$


 The following proposition is the key to extend Theorem \ref{AU} to the infinite setting (see also  \cite[Lemma 3.2]{CSZ},  \cite[Proposition 19]{SZ} and \cite[Lemma 3.2]{BM} for similar results).


   \begin{proposition}\label{partition}
If $x\prec y\in L_1(0,\infty)$, then there exist sequences of disjoint sets $\{A_n\subset \bR_+\oplus \bR_+ \}$ and $\{B_n \subset \bR_+ \oplus \bR_+ \}$
such that $m(A_n)=m(B_n)<\infty$, $\cup A_n =\cup B_n=\bR_+\oplus \bR_+ $  and $(x\oplus 0 )\chi_{B_n}\prec (y\oplus 0)\chi_{A_n}$ (on $\bR_+\oplus \bR_+$) for every $n\ge 0$.
  \end{proposition}
 \begin{proof}
Let
$$\sigma_+^x:s(x_+)\to (0, m(s(x_+)) ),~ \sigma_+^y:s(y_+)\to (0, m(s(y_+)) ),$$
$$\sigma_-^x:s(x_- )\to (0, m(s(x_-)) ),~\sigma _-^y:s(y_-)\to (0, m(s(y_-)) ) $$
be measure-preserving transformations such that
$$x_+ =\lambda(x_+)\circ \sigma_+^x,~ x_- =\lambda(x_-)\circ \sigma_-^x, ~ y_+ =\lambda(y_+)\circ \sigma_+^y, ~y_- =\lambda(y_-)\circ \sigma_-^y.$$


Define
$$X_1 :=\{ s: \mu(s; y_+)>\mu(s;x_+) \},$$
$$X_2 :=\{ s: \mu(s; y_+)<\mu(s;x_+) \},$$
$$X_3 :=\{ s: \mu(s; y_+)=\mu(s;x_+) \}.$$
Let
 $\{\Gamma_k^2:  m(\Gamma_k^2)<\infty \}_{k\ge 1}$
be a partition  of $X_2$.
Since $\tau(y_+)\ge \tau(x_+)$,
it follows that there exists a collection  $\{\Gamma_k^1: m(\Gamma_k^1)<\infty \}_{k\ge 1}$ of disjoint subsets of $X_1$ such that
$$\int_{\Gamma_k^1 \cup \Gamma_k^2} \mu(y_+)- \mu(x_+)=0.$$
We denote by $\Gamma_0^1 := X_1\setminus (\cup_{k\ge 1} \Gamma_k^1)$ (it may be the empty set).
We also define
 a partition  $\{\Gamma_k^3:  m(\Gamma_k^3)<\infty \}_{k\ge 1}$ of $X_3$.
 For brevity, we numerate $\{\Gamma_k^1 \cup \Gamma_k^2\}_{k\ge 1} \cup \{\Gamma_k^3\}_{k\ge 1}$ as $\{\Delta_{k}\}_{k\ge 1}$ and
let $A_0:=\Gamma_0^1 $.

Similarly,
we construct
a partition  $\{\Omega_k^1: m(\Omega_k^1)<\infty \}_{k\ge 1}$ of $\{ s: \mu(s; y_-)>\mu(s;x_-) \}$ and a decomposition $\{\Omega_k^2: m(\Omega_k^2)<\infty \}_{k\ge 1}$ of $\{ s: \mu(s; y_-)>\mu(s;x_-) \}$  such that
$$\int_{\Omega_k^1 \cup \Omega_k^2} \mu(y_-) -\mu(x_-)=0 .$$
We denote by
$$\Omega_0^1 := \{  s: \mu(s;y_-) > \mu(s;x_-)\}\setminus (\cup_{k\ge 1} \Omega_k^1)$$ (it may be the empty set).
We also define
 a partition  $\{\Omega_k^3:
m(\Omega_k^3)<\infty \}_{k\ge 1}$ of $\{ s: \mu(s; y_-)=\mu(s;x_-) \}$.
For brevity, we numerate $\{\Omega_k^1 \cup \Omega_k^2\}_{k\ge 1} \cup \{\Omega_k^3\}_{k\ge 1}$ as $\{\Lambda_{k}\}_{k\ge 1}$ and let   $B_0:=\Omega_0^1   $.

We consider
couples
$$
\Big((\sigma_ +^x)^{-1}  ( \Delta_k \cap (0,m(s(x_+)))) , ~  (\sigma_ +^y)^{-1}  (\Delta_k \cap (0,m(s(y_+))))
\Big) $$
and
$$
\Big((\sigma_ -^x)^{-1}   ( \Lambda_k \cap (0,m(s(x_-)))),  ~(\sigma_-^y )^{-1}(\Lambda_k \cap (0,m(s(y_-))))\Big) .
$$
Note that the measures of elements in any couple above may not be the same.
We can find sequences  $\{\Delta_k^{cx}\}$, $\{\Delta_k^{cy}\}$, $\{\Lambda_k^{cx}\}$ and  $\{\Lambda _k^{cy}\}$ of disjoint  sets of finite measures in $(0,\infty)$ such that
$$(m\oplus m) ((\sigma_ +^x)^{-1}  ( \Delta_k \cap (0,m(s(x_+))))\oplus \Delta_k^{cx}) = (m \oplus m) ( (\sigma_ +^y)^{-1} (\Delta_k \cap (0,m(s(y_+)))) \oplus  \Delta_k^{cy}),$$
and
$$(m \oplus m) ((\sigma_ -^x)^{-1}  ( \Lambda_k \cap (0,m(s(x_-))))\oplus \Lambda_k^{cx}) = (m \oplus m) ( (\sigma_ -^y)^{-1} (\Lambda_k \cap (0,m(s(y_-)))) \oplus  \Lambda_k^{cy}).$$

Now, consider $A_0:=(0,\infty)\setminus\left(\cup _k \Delta_k\right)$ and $B_0:=(0,\infty)\setminus\left(\cup _k \Lambda_k\right) $.
Since $\tau(y)=\tau(x)$,
it follows that there exist partitions  $\{\cX^0_k:m(\cX^0_k)<\infty \}$ and $\{\cY^0_k:m(\cY^0_k)<\infty \}$ of $A_0$ and $B_0$, respectively,
such that
$$ \int_{\cX_k^0} \mu(y_+) -\mu(x_+) =\int_{\cY^0_k}\mu(y_-) -\mu(x_-) .$$
We consider
couples
$$
\left((\sigma_ +^x)^{-1}  ( \cX_k^0 \cap (0,m(s(x_+)))) , ~  (\sigma_ +^y)^{-1} (\cX_k^0 \cap (0,m(s(y_+))))
\right) $$
and
$$
\left( (\sigma_-^x)^{-1}   (\cY^0_k \cap (0,m(s(x_-)))) ,  ~(\sigma_-^y )^{-1} (\cY^0_k \cap (0,m(s(y_-))))
\right) .
$$
Note that the measures of elements in any couple may not  be the same.
We can find sequences  $\{\cA_k^{cx}\}$, $\{\cA_k^{cy}\}$, $\{\cB_k^{cx}\}$ and $\{\cB_k^{cy}\}$ of disjoint sets of finite measures in $(0,\infty)$ such that
$$
(m \oplus m )\left(
(\sigma_ +^x)^{-1}  ( \cX_k^0 \cap (0,m(s(x_+))))\oplus \cA_k^{cx}  \right)  = (m \oplus m ) \left( (\sigma_ +^y)^{-1} (\cX_k^0 \cap (0,m(s(y_+))))\oplus \cA_k^{cy}
\right) $$
and
$$
(m \oplus m) \left(
(\sigma_ -^x)^{-1}  ( \cY_k^0 \cap (0,m(s(x_-))))\oplus \cB_k^{cx}  \right)  = (m \oplus m) \left((\sigma_ -^y)^{-1} (\cY_k^0 \cap (0,m(s(y_-))))\oplus \cB_k^{cy}
\right).
$$
Now, we numerate the following disjoint couples of subsets in $\oplus_{k=1}^4  (0,\infty)$ by $(X_k,Y_k)$,
$$\Big( (\sigma_ +^x)^{-1}  ( \Delta_k \cap (0,m(s(x_+))))\oplus \Delta_k^{cx} \oplus \varnothing \oplus  \varnothing  ,\quad (\sigma_ +^y)^{-1}  ( \Delta_k \cap (0,m(s(y_+))))\oplus \Delta_k^{cy} \oplus \varnothing \oplus  \varnothing   \Big),$$
$$ \Big ( (\sigma_ -^x)^{-1}  ( \Lambda_k \cap (0,m(s(x_-))))\oplus \Lambda_k^{cx}\oplus \varnothing\oplus  \varnothing, \quad(\sigma_ -^y)^{-1}  ( \Lambda_k \cap (0,m(s(y_-))))\oplus \Lambda_k^{cy}\oplus \varnothing\oplus  \varnothing   \Big), $$
\begin{align*}
\Big(
&(\sigma_ +^x)^{-1}  ( \cX_k^0 \cap (0,m(s(x_+)))) \cup (\sigma_ -^x)^{-1}  ( \cY_k^0 \cap (0,m(s(x_-))))\oplus  \varnothing\oplus  (\cA_k^{cx}\cup \cB_k^{cx} ) \oplus  \varnothing ,\\
&
(\sigma_ +^y)^{-1}  ( \cX_k^0 \cap (0,m(s(y_+)))) \cup
(\sigma_ -^y)^{-1}  ( \cY_k^0 \cap (0,m(s(y_-))))\oplus  \varnothing\oplus (\cA_k^{cy} \cup \cB_k^{cy}) \oplus  \varnothing
\Big).
\end{align*}
In particular,  the measures of $X_k$ and $Y_k$ are the same.

Since there exists a measure-preserving isomorphism  between
$(0,\infty)$ and
$$(0,\infty)^{\oplus  ^3}\oplus ((0,\infty )\setminus s(x))~\mbox{~(or $(0,\infty)^{\oplus  ^3} \oplus (0,\infty )\setminus s(y)$)},$$
  we may
view $(X_k)$ (and $(Y_k)$) as disjoint subsets of $(0,\infty)\oplus (0,\infty)$.
Moreover, since the complements  $(\cup (X_k))^c$ (resp. $(\cup (X_k))^c$) of $\cup (X_k)$ (resp. $\cup (Y_k)$) has infinite measure,
it follows that there
exists a  partition $\{P_n\}$ (resp. $\{Q_n\}$) of  $(\cup (X_k))^c$ (resp. $(\cup (Y_k))^c$)  with $(m \oplus m ) (P_n) =1 $
(resp. $(m \oplus m ) (Q_n) =1 $).

Numerating $(X_k)$ and $(P_n)$ (resp. $(Y_k)$ and $(Q_n)$),  we complete the proof.
\end{proof}


Now, we prove the
equivalence between (b) and (d) in Theorem \ref{main}.
  \begin{proposition}\label{btod}
  Let $x,y\in L_1(0,\infty)$.
  Then,
  $x\prec y$ if and only if $\tau(f(x))\le \tau(f(y))$ for any convex function $f$ on the real axis with $f(0)=0$ such that $f(x)$ and $f(y)$ are integrable.
  Here, $\tau(f)$ stands for $\int f dm$ ($m$ is the Lebesgue measure).
  \end{proposition}
\begin{proof}$\Rightarrow$.
By Proposition \ref{partition}, there exist sequences of disjoint sets
$$\{A_n\subset \bR_+ \oplus \bR_+ \}\mbox{ and  }\{B_n \subset   \bR_+\oplus \bR_+ \}$$
such that $m(A_n)=m(B_n)<\infty$, $\cup A_n =\cup B_n =\bR_+\oplus \bR_+$  and $(x\oplus 0)\chi_{B_n}\prec (y\oplus 0) \chi_{A_n}$ for every $n$.
By the classic result (see e.g. \cite[Proposition 1.3]{HN2}), we have
$$(\tau\oplus\tau)(f( (x\oplus 0) \chi_{B_n}) )\le (\tau\oplus\tau)(f( (y\oplus 0)\chi_{A_n}))$$ for every $n$.
Hence, we have
$$ (\tau\oplus\tau)(f(x\oplus 0)) =\sum_n (\tau\oplus\tau)(f( (x\oplus 0) \chi_{B_n}) ) \le \sum _n (\tau\oplus\tau)(f( (y\oplus 0)\chi_{A_n}))=(\tau\oplus\tau)(f(y\oplus 0 )) .$$
Since $f(0)=0$, it follows that $\tau(f(x))\le \tau(f(y))$.

$\Leftarrow$.
Let $f(t)=t$.
Then, we have $\tau(x)\le \tau(y)$.
Let $f(t)=-t$.
We have $\tau(-x)\le \tau(-y)$.
Hence, $\tau(x)=\tau(y)$.
On the other hand, we take any non-decreasing continuous convex function $f$ on $\bR_+$ with $f(t)=0$ when $t<0$.
We have $\tau(f(x_+))\le \tau(f(y_+))$.
Hence, $x_+ \prec \prec y_+$ (see \cite[Proposition 1.2]{HN2}).
Now, we take any non-decreasing continuous convex function $g$ on $\bR_+$ with $g(t)=0$ when $t<0$.
We define  $f(t)=g(-t)$ on $\bR $.
We have
$$\tau(g(x_-))=\tau(f(-x_-))=\tau(f(x)) \le \tau(f( y )) =\tau(f(-y_-)) = \tau(g(y_-)) .$$
Hence,
$\tau(g(x_-))\le \tau(g(y_-))$,
i.e.,
$x_-\prec \prec y_-$ (see \cite[Proposition 1.2]{HN2}).
\end{proof}

\section{Alberti--Uhlmann problem in the setting of  infinite von Neumann algebras}
Let $L_1(\cM,\tau)$ be the noncommutative $L_1$-space affiliated with a semifinite von Neumann algebra $\cM$ equipped with a faithful normal semifinite trace $\tau$.
Throughout this section, we always assume that  $\tau({\bf 1})=\infty $.

The following proposition for  positive operators can be found in \cite[Theorem 4.5]{Hiai}.
  \begin{proposition}\label{majtraceequ}
 Let $y\in L_1(\cM,\tau)_h$ and
  let $\varphi \in DS(\cM)$.
  Then, $\varphi(y)\prec y$.
  \end{proposition}
  \begin{proof}
 Since $\varphi(y) -s =\varphi(y-s) \le \varphi((y-s)_+)$ for every $s\in \bR$,
 it follows that
 $$\tau(  (\varphi(y) -s)_+)= \tau(p (\varphi(y) -s)p) \le \tau(p\varphi((y-s)_+)p )\le \tau(  \varphi((y-s)_+)  )=\tau((y-s)_+)$$ for every $s\ge  \bR_+$,
 where $p:=s( (\varphi(y) -s)_+)$.
 The same argument shows that
  $$\tau((-\varphi(y) -s)_+) \le \tau(\varphi((-y-s)_+))=\tau((-y-s)_+)$$ for every $s\ge  \bR_+$.
 Since $\varphi$ is trace-preserving, it follows from  Proposition \ref{treq} that the assertion holds.
  \end{proof}


Now, we present the proof of (b)$\Rightarrow$(a) in Theorem \ref{main}.
  \begin{theorem}\label{stoinfi}
  If $\cM$ is a $\sigma$-finite von Neumann algebra, then
  for any $x\prec y\in L_1(\cM,\tau)_h$, there exists
a normal doubly stochastic operator $\varphi:\cM \oplus L_\infty(0,\infty)  \to \cM \oplus L_\infty (0,\infty)$ such that
$\varphi(y\oplus 0)=x\oplus 0 $.
  \end{theorem}

\begin{proof}
Without loss of generality, we may assume that $\cM$ is atomless \cite[Lemma 2.3.18]{LSZ}.
For every $x\in \Omega (y)$,
there exists an abelian non-atomic von Neumann subalgebra $\cN_1^x$ of $s(x_+)\cM s(x_+)$
containing all spectral projections of $x_+$, and a $*$-isomorphism $V_1$ from $S(\cN_1^x,\tau)$ onto $S([0,\tau(s(x_+)),m)$ such that
$\lambda (V_1(f))  =\lambda (f)$ for every $f\in S(\cN_1^x,\tau)_h$  (see \cite[Lemma 1.3]{CKS}, see also \cite[Proposition 3.1]{BKS1}, \cite{CK} and  \cite[Lemma 4.1]{CS}).
Similarly, we have a $*$-isomorphism $V_2$ from $S(\cN_2^x,\tau)$ onto $S([0,\tau(s(x_-)),m)$ such that
$\lambda (V_2(f))  =\lambda (f)$ for every $f\in S(\cN_2^x,\tau)_h$,
where $\cN_2^x$ is  an abelian non-atomic von Neumann subalgebra  of $s(x_-)\cM s(x_-)$
containing all spectral projections of $x_-$.
We denote  $$V_x:= V_1 ^x \oplus V_2^x.$$
There exists a commutative   atomless von Neumann subalgebra $\cN_3^x $ of $\cM_{n(x)}$, on  which $\tau$ is again semifinite.

The same argument show that there exists such a trace-preserving $*$-isomorphism $V_y:  S(\cN_1^y\oplus \cN_2^y) \to S(0, \tau(s(y)))$,
where $\cN_1^y$ (resp. $\cN_2^y$) is an atomless commutative reduced  von Neumann subalgebra of $\cM_{s(y_+)}$ (resp. $\cM_{s(y_-)}$) containing all spectral projections of $y_+$ (resp. $y_-$).
$\cN_3^y$ denotes a commutative atomless von Neumann subalgebra of $\cM_{n(y)}$.
There exists a trace-preserving normal conditional expectation $\cE:\cM\to \cN_1^y \oplus \cN_2^y\oplus \cN_3^y $  \cite[Proposition 2.1]{DDPS}.

 Since $\cM$ is $\sigma$-finite, it follows that  $\cN_3$ is isomorphic to a subalgebra of $L_\infty (0,\infty )$.
Hence, there are trace-preserving $*$-isomorphisms $U_1$ and $U_2$ such that
$$   L_\infty (0, \tau(s(x)) ) \oplus L_\infty (0,\infty ) \backsimeq_{U_1} \cN_1^{x}\oplus \cN_2^{x} \oplus \cN_3^{x} \oplus L_\infty (0,\infty ) $$
and
$$   L_\infty (0, \tau(s(y)) ) \oplus L_\infty (0,\infty ) \backsimeq_{U_2} \cN_1^{y}\oplus \cN_2^{y} \oplus \cN_3^{y} \oplus L_\infty (0,\infty ) \subset^\cE \cM \oplus L_\infty (0,\infty) . $$
Hence, without loss of generality, we may assume that
$y\in  L_1 (0, \tau(s(y)) ) $, $x\in  L_1 (0,\tau(s(x)) )$ and it suffices to show that there exists a doubly stochastic operator $\varphi$ from  $L_\infty (0, \tau(s(y)) ) \oplus L_\infty (0,\infty )$ to $L_\infty (0,\tau(s(x)) ) \oplus L_\infty (0,\infty ) $ whose $L_1$-extension maps $y$ to $x$.

By Proposition  \ref{partition}, there exist
 $\{A_n\subset (0,\tau(s(y)))\oplus \bR_+ \}$
and
$\{B_n \subset (0,\tau(s(x)))  \oplus \bR_+ \}$
such that
$m(A_n)=m(B_n)<\infty$,
$\cup A_n =(0,\tau(s(y)))\oplus \bR_+$
and
$\cup B_n= (0,\tau(s(x)))\oplus \bR_+  $  and $(x\oplus 0) \chi_{B_n}\prec (y\oplus 0)\chi_{A_n}$  for every $n\ge 0$.
By Theorem \ref{th:f} (see also \cite{Day}),
there are normal doubly  stochastic operators $\varphi_n :L_\infty (B_n) \to L_\infty (A_n)$ with $\varphi_n (y_{A_n})=x_{B_n}$.
By taking the direct sum $\Phi$ of all $\varphi_n$, we obtain a  doubly stochastic operator
$$ \Phi(=\oplus_{n\ge 1}\varphi_n ):=L_\infty (0,\tau(s(x)))  \oplus L_\infty (0,\infty ) \to  L_\infty (0,\tau(s(y)))  \oplus L_\infty (0,\infty )  $$
such that $\Phi(y)=x$.
Moreover,
since every $\varphi_n$ is normal, it follows that the direct sum $\Phi$ is also normal.
Indeed,  let $z_i \uparrow z\in L_\infty (0,\tau(s(y)))  \oplus L_\infty (0,\infty ) $.
For any $n\ge 1$,  we have
$$\varphi_n(z_i \chi_{A_n}) \uparrow \varphi(z\chi_{A_n}).$$
Since
$$\Phi(z_i)=\oplus_{n\ge 1}\varphi_n(z_i\chi_{A_n})$$
and
$$\Phi(z)=\oplus_{n\ge 1}\varphi_n(z\chi_{A_n}),$$
it follows that
 that $\Phi(z_i)\uparrow \Phi(z)$, which implies that $\Phi$ is normal.
\end{proof}

It is interesting to see that
 Hiai's conjecture is not true in the setting of non-$\sigma$-finite semifinite von Neumann algebras.
 The proof is similar with that of Theorem~\ref{stoinfi}.
Before proceeding to the proof of this result, we need the following lemma.
We would like to thank Dr. Dmitriy Zanin for his help with  the proof of following lemma.
\begin{lemma}\label{lemma:z}
Let $\cA$ be an  atomless abelian semifinite infinite von Neumann algebra equipped with a semifinite faithful normal trace $\tau$.
If $\cA$ is not $\sigma$-finite, then
\begin{align}\label{isoinfi}
\cA_{{\bf 1}-s(x)} \backsimeq \cA
\end{align}
for any $x\in S_0(\cA,\tau)$.
Here, we denote by $\cA_p$ the reduced von Neumann subalgebra $p\cA p$ of $\cA$, $p\in \cP(\cA)$.
\end{lemma}
\begin{proof}
Since $x$ is $\tau$-compact, it follows   that $s(x)$ is a $\sigma$-finite projection, and the atomless commutative algebra $\cA_1:= \cA_{s(x)}$ is isomorphically isomorphic (the $*$-isomorphism is denoted by $U_1$) to $L_\infty (0,s(x))$ \cite[Theorem 9.3.4]{Bogachev}. 

Since  $ \cA_{{\bf 1} -s(x)}$ is not $\sigma$-finite,
we can construct a commutative subalgebra $\cA_2$ of $ \cA_{{\bf 1} -s(x)}$ is isomorphically isomorphic to $L_\infty (0,s(x))$.

Arguing inductively, we may construct a sequence $\{\cA_n\}$ of disjoint commutative subalgebras  of $\cA$,  isomorphically isomorphic to $L_\infty (0,\tau(s(x)))$.
We denote by $U_n$ the $*$-isomorphism between $ \cA_{{\bf 1} -s(x)}$ and $\cA_n$.

Let $\cA_0:= \cA_{{\bf 1}-\oplus_n {\bf 1}_{\cA_n}}$.
Since ${\bf 1}$ is not $\sigma$-finite in $\cA$, it follows that  ${\bf 1}-\oplus_n {\bf 1}_{\cA_n}$ is not $\sigma$-finite.
Since $\cA_0$ is a reduced algebra of $\cA$, it follows that the restriction of $\tau$ on $\cA_0$ is semifinite,  faithful and normal.

Now, we can define a trace-preserving $*$-isomorphism $\alpha: \oplus _{n\ge 0}\cA_n \to \oplus _{n\ne 1 }\cA_n $ by setting
$$  \alpha(x)=\left\{
                \begin{array}{ll}
                  x, ~ x\in \cA_0, \\
                 U_{n+1}( U_n^{-1}(x)),~ x\in \cA_{n},~n\ge 1 .
                \end{array}
              \right.$$
\end{proof}

 \begin{theorem}\label{stopureinf}
  If $\cM$ is not a $\sigma$-finite von Neumann algebra, then
  for any $x\prec y\in L_1(\cM,\tau)_h$, there exists
a normal doubly stochastic operator $\varphi:\cM \to \cM $ such that
$\varphi(y)=x $.
  \end{theorem}

\begin{proof}
Without loss of generality, we may assume that $\cM$ is atomless \cite[Lemma 2.3.18]{LSZ}.
Note that $x$ and $y$ are $\tau$-compact.
For every $x\in \Omega (y)$,
there exists an abelian atomless von Neumann subalgebra $\cN_1^x$ of $s(x_+)\cM s(x_+)$
containing all spectral projections of $x_+$, and a $*$-isomorphism $V_1$ from $S(\cN_1^x,\tau)$ onto $S([0,\tau(s(x_+)),m)$ such that
$\lambda (V_1(f))  =\lambda (f)$ for every $f\in S(\cN_1^x,\tau)_h$  (see \cite[Lemma 1.3]{CKS}, see also \cite[Proposition 3.1]{BKS1}, \cite{CK} and  \cite[Lemma 4.1]{CS}).
Similarly, we have a $*$-isomorphism $V_2$ from $S(\cN_2^x,\tau)$ onto $S([0,\tau(s(x_-)),m)$ such that
$\lambda (V_2(f))  =\lambda (f)$ for every $f\in S(\cN_2^x,\tau)_h$,
where $\cN_1^x$ is  an abelian atomless von Neumann subalgebra  of $s(x_-)\cM s(x_-)$
containing all spectral projections of $x_-$.
We denote  $$V_x:= V_1 ^x \oplus V_2^x.$$

The same argument show that there exists such a trace-preserving $*$-isomorphism $V_y:  S(\cN_1^y\oplus \cN_2^y) \to S(0, \tau(s(y)))$,
where $\cN_1^y$ (resp. $\cN_2^y$) is an atomless commutative reduced  von Neumann subalgebra of $\cM_{s(y_+)}$ (resp. $\cM_{s(y_-)}$) containing all spectral projections of $y_+$ (resp. $y_-$).

Since $\cM$ is not $\sigma$-finite, it follows that there exists an  atomless non-$\sigma$-finite commutative von Neumann subalgebra $\cA$ of $\cM_{{\bf 1}- s(x)\vee s(y)}$.
It follows
 from Lemma~\ref{lemma:z} that
  $$\cA \simeq \cA \oplus L_\infty (0,a)  $$
  for any $0< a\le \infty $.

Note that   $s(x)\vee s(y)-s(y)$ is a $\sigma$-finite projection. 
 Let $\cB_y$ (resp. $\cB_x$) be a commutative atomless von Neumann subalgebra of $\cM_{s(x)\vee s(y) - s(y)}$ (resp. $\cM_{s(x)\vee s(y) - s(x)}$).
 We have $$\cA \oplus \cB_y \simeq \cA$$
 (resp. $\cA \oplus \cB_x \simeq \cA$).
There exists a trace-preserving conditional expectation $\cE:\cM\to \cN_1\oplus \cN_2\oplus \cA   \oplus \cB_y $ \cite[Proposition 2.1]{DDPS}.
Hence, there are trace-preserving isomorphisms $U_1$ and $U_2$ such that
\begin{align*}
  L_\infty (0, \tau(s(x))  ) \oplus  L_\infty (0,\infty )\oplus  \cA &\backsimeq_{U_1} \cN_1^x\oplus \cN_2^x \oplus \cA \oplus \cB_x \subset \cM
  \end{align*}
and
\begin{align*} L_\infty (0, \tau(s(y))  ) \oplus  L_\infty (0,\infty )\oplus  \cA  &\backsimeq_{U_2 } \cN_1^y\oplus \cN_2 ^y \oplus \cA  \oplus \cB_y \subset^\cE \cM
\end{align*}
where the restrictions of $U_1$ and $U_2$ coincide with $V_x$ and $V_y$.
Hence, without loss of generality, we may assume that
$y\in  L_1 (0,\tau(s(y)) ) $, $x\in  L_1 (0, \tau(s(x))  )$. The argument in Theorem \ref{stoinfi} infers that
 there exists a normal doubly stochastic operator $\varphi$ from  $L_\infty (0, \tau( s(y) )  ) \oplus L_\infty (0,\infty )$ to $L_\infty (0, \tau(s(x))  ) \oplus L_\infty (0,\infty ) $ whose $L_1$-extension  maps $y$ to $x$.
\end{proof}

\section{Extreme points of $\Omega (y)$: Luxemburg's problem in the infinite setting}\label{secCom}
In this section,
we consider Luxemburg's problem \cite[Problem 1]{Luxemburg} in the (noncommutative) infinite setting,
  characterizing  the extreme points of $\Omega(y)$, $y\in L_1(\cM,\tau)_h$ when $\cM$ is a semifinite von Neumann equipped with a semifinite infinite faithful normal trace $\tau$.
The main idea  steming  from \cite{DHS} still works for the infinite setting.
However,
the description of  extreme points of $\Omega(y)$ is much more complicated than that in the finite setting.
We show that the   case of an  infinite von Neumann algebra can be reduced into
 the case of at most countably many finite von Neumann algebras and we can apply the same idea used in   \cite{DHS}.

The theorem below is the main result in this section.
\begin{theorem}\label{th:ext}
Let $y\in L_1(\cM,\tau)_h$.
 Then,   $x\in L_1(\cM,\tau)$ is an extreme point of $\Omega (y)$ if and only if  for $x_+$ and $y_+$,  for any $t\in  (0,\infty)$, one of the following options holds:
       \begin{enumerate}
  \item[(i).]   $\lambda (t;x)=  \lambda(t; y)$;
  \item [(ii).]
   $\lambda(t;x) \ne \lambda(t;y)$ with  the spectral projection $ E^x \{\lambda (t;x)  \} $ being an atom in $\cM$ and $$\int_{ \{s;\lambda (s;x)=\lambda (t;x)\}}  \lambda (s;y)ds    =  \lambda(t ;x)     \tau(E^x (\{\lambda (t;x)\}))  .  $$
\end{enumerate}
and,   for $x_-$ and $y_-$, for any $t\in  (0,\infty)$,
one of the following options holds:
       \begin{enumerate}
  \item[(i).]   $\lambda (t;-x)=  \lambda(t; -y)$;
  \item [(ii).]
   $\lambda(t;-x) \ne \lambda(t;y)$ with  the spectral projection $ E^{-x} \{\lambda (t;-x)  \} $ being an atom in $\cM$ and $$\int_{ \{s;\lambda (s;-x)=\lambda (t;-x)\}}  \lambda (s;-y)ds    =  \lambda(t ;-x)     \tau(E^{-x} (\{\lambda (t;-x)\}))  .  $$
\end{enumerate}
\end{theorem}

The following lemma allows us to reduce the problem to the setting of the positive core of a finite von Neumann algebra.
\begin{lemma}\label{lemma:exeq}
Let $y\in L_1(\cM,\tau)_h $.
If $x\in L_1(\cM,\tau)_h$ is an extreme point of $\Omega(y)$,
then
$\tau(x_+)=\tau(y_+)$
and $\tau(x_-)=\tau(y_-)$.
\end{lemma}
\begin{proof}
Assume by contradiction  that $\tau(x_+)<\tau(y_+)$.
Since $\tau(x_+)-\tau(x_-)=\tau(x)=\tau(y)=\tau(y_+)-\tau(y_-)$, it follows that $\tau(x_-)<\tau(y_-)$.
Let
 $$
 A:=\left\{t\in (0,\infty ]:  \int_0^{t} \lambda (t;y_+ )dt=\int_0^{t} \lambda (t;x_+)dt\right\} .
 $$
 If $\sup A =\infty$,
then $\tau(y_+)=\tau(x_+)$.
Hence, $\sup A<\infty $.
We define $t': =\sup A$.
In particular,  $t'\in A$.
Similarly, we define $B:=\left\{t:  \int_0^{t} \lambda (t;y_- )dt=\int_0^{t} \lambda (t;x_-)dt\right\}$ and $B\ni s':= \sup B<\infty $.
For any $t> t'$, we have
$$\int_0^{t} \lambda (t;y_+ )dt> \int_0^{t} \lambda (t;x_+)dt $$
and for any
$t>s'$, we have
$$\int_0^{t} \lambda (t;y_- )dt> \int_0^{t} \lambda (t;x_- )dt . $$
By Lemma \ref{counter example 3}, $\lambda (x_\pm)$ is a step function on $(t',\infty)$ (resp. $(s',\infty )$).
Assume that $\lambda(x_+)>0$ on $\bR_+$.
Since $x$ is $\tau$-compact, it follows $\lambda(x_+)$ decreases to $0$ at infinity.
By Lemma \ref{counter example 3}, $x$ is not an extreme point.
Hence, $\lambda (x_+)$ has a finite support.
The same argument shows that $\lambda (x_-)$ has a finite support.


Recall that
$$\int_0^{t'} \lambda (t;y_+ )dt=\int_0^{t'} \lambda (t;x_+)dt, $$
$$\int_0^{s'} \lambda (t;y_- )dt=\int_0^{s'} \lambda (t;x_-)dt .$$
We claim that
\begin{align}\label{yxright}
0<\lambda (t'-;y_+ ) \le  \lambda (t'-;x_+)
\end{align}
and
$$0<  \lambda (s'-;y_- ) \le  \lambda (s'-;x_-) .$$
Otherwise, $ \lambda (t'-;y_+ ) >  \lambda (t'-;x_+) $, i.e., there exists $\delta>0$ such that $ \lambda (y_+ ) >  \lambda (x_+) $ on $(t'-\delta, t')$.
Since $x\prec y$, it follows that $\int_0^{t'} \lambda(t;x_+)dt< \int_0^{t'} \lambda (t;y_+)dt$,
which is a contradiction with $t'\in A$.
The case for $x_-$ and $y_-$ follows from the same argument.

We claim that for any $\varepsilon>0$,
\begin{align}\label{t'x+}
\lambda (t'+\varepsilon; x_+)< \lambda (t'-;x_+)
\end{align}
and
$$\lambda (s'+\varepsilon; x_-)< \lambda (s'-;x_-).$$
Otherwise,
$\lambda (t'+\varepsilon; x_+)=\lambda (t'-;x_+) \stackrel{\eqref{yxright}}{\ge}  \lambda (t'-;y_+ ) \ge \lambda (t'+\varepsilon;y_+ ) $ for some $\varepsilon>0$.
That is,
$$\int_{t'}^{t'+\varepsilon} \lambda (s;x_+)ds \ge  \int_{t'}^{t'+\varepsilon} \lambda (s;y_+)ds.  $$
This contradicts with $t'=\sup A$ (if we take ``$=$'' in the above inequality) or $x\prec y$ (if we take ``$>$'' in the above inequality).
The case for $x_-$ follows from the same argument.

Let's consider $\lambda (t';x_+)$.
There are three possible situations:
\begin{enumerate}
  \item $\lambda (t';x_+)=0$;
  \item $\lambda (t';x_+    )>0$ and $\lambda (t'+\varepsilon;x_+)=\lambda (t';x_+    )$ for some $\varepsilon>0$;
  \item $\lambda (t';x_+    )>0$ and $\lambda (t'+\varepsilon;x_+)<\lambda (t';x_+    )$ for any $\varepsilon>0$.
\end{enumerate}
By Lemma \ref{counter example 2} and the right-continuity of $\lambda (x_+)$, situation (3) is impossible.
Consider situation (2).
We claim that  $E^{x_+} (0, \delta)=0$ for some sufficiently small  $\delta >0$.
Otherwise,
$\lambda (x_+)$ satisfies the conditions in Lemma \ref{counter example 3}.
That is, $x$ is not an extreme point of $\Omega (y)$.
Hence,
$m_1:= \inf \{\lambda (x_+)>0\} >0$.
Similarly, we obtain that
$m_2:= \inf \{\lambda (x_-)>0\} >0$.
If situation (1) is true, then, by \eqref{t'x+}, we can define $m_1>0$ by $\lambda (t'-;x_+)$ (resp. $m_2:= \lambda (s'-;x_-)>0$).

Recall that $s(x)$ is $\tau$-finite and the trace  $\tau$ is infinite.
Take any non-zero $\tau$-finite projections $P_1,P_2\le 1-s(x)$ and $P_1\perp P_2$.
Without loss of generality, we may assume that $\tau(P_1)\le \tau(P_2)$.
Let
$$C_1 := \frac{1}{\tau(P_1)}\left(\int_{0}^{\tau(s(x_+))+\tau(P_1)}\lambda (t;y_+)dt- \int_{0}^{\tau(s(x_+)) }\lambda (t;x_+)dt\right)>0,$$
$$C_2 := \frac{1}{\tau(P_1)} \left(
\int_{0}^{\tau(s(x_-))+\tau(P_1)}\lambda (t;y_-)dt -\int_{0}^{\tau(s(x_-)) }\lambda (t;x_-)dt\right)>0  ,$$
$$C_3 := \frac{1}{\tau(P_2)}\left(\int_{0}^{\tau(s(x_+))+\tau(P_2)}\lambda (t;y_+)dt- \int_{0}^{\tau(s(x_+)) }\lambda (t;x_+)dt\right)>0,$$
and
$$C_4 := \frac{1}{\tau(P_2)} \left(
\int_{0}^{\tau(s(x_-))+\tau(P_2)}\lambda (t;y_-)dt -\int_{0}^{\tau(s(x_-)) }\lambda (t;x_-)dt \right)>0 . $$
There exists a $\delta >0$ such that
$$\delta < \min\left\{ C_1,C_2,C_3,C_4, m_1,m_2  \right\}$$
We define
$$
x_1 := x + \delta \cdot \left(P_1 - \frac{\tau(P_1)}{\tau(P_2)} P_2\right)
 $$
and
$$
x_2:=x-\delta \cdot \left( P_1 - \frac{\tau(P_1)}{\tau(P_2)} P_2\right ).
  $$
Clearly, $\tau(x_1)=\tau(x_2)=\tau(x)=\tau(y)$ and  and $\frac{x_1+x_2}{2}=x$.
Note that
$t\mapsto \int_0^t\lambda(s;y)ds $ is a  concave function  and
$$ \int_0^{\tau(s(x_+))}\lambda(s;y)ds \ge   \int_0^{\tau(s_+)} \lambda(s;(x_+) )ds =\int_0^{\tau(s_+)} \lambda(s;(x_1) )ds ,$$
$$ \int_0^{\tau(s(x_+))+\tau(P_1)}\lambda(s;y)ds >  \int_0^{\tau(s(x_+))+\tau(P_1)} \lambda(s;x_1)ds=\int_0^{\tau(s_+)} \lambda(s;x_1 )ds  +\delta \tau(P_1).$$
Moreover, by spectral theorem, we obtain that  $\lambda((x_1)_+)$ is a positive constant on $[ \tau(s(x_+)),\tau(s(x_+))+\tau(P_1))$ and vanishes on $[\tau(s(x_+))+\tau(P_1),\infty)$.
Hence, $\lambda ((x_1)_+)\prec \prec \lambda (y)$.
Arguing similarly to $(x_1)_-$, $(x_2)_+$ and $(x_2)_-$,
we obtain that $ x_1,x_2\prec y$,
which shows that $x$ is not an extreme point of $\Omega(y) $.
Hence, $\tau(x_+)=\tau(y_+)$ and $\tau(x_-)=\tau(y_-)$.
\end{proof}

Now, we present the proof of Theorem \ref{th:ext}.

\begin{proof}[Proof of Theorem \ref{th:ext}]
``$ \Rightarrow$''.  Assume that $x\in \text{extr}(\Omega(y))$.
By Lemma \ref{lemma:exeq}, we obtain that $\tau(x_+)=\tau(y_+)$ and $\tau(x_-)=\tau(y_-)$.
It suffices to prove the case for $x_+$ and $y_+$.
Hence, we may always assume that
$$x,y\ge 0 .$$

We set
$$A=\left\{s\in (0,\infty):\int_0^s\lambda (t;y)-\lambda (t; x)dt>0\right\}.$$
Since $\lambda (y),\lambda (x)\in L_1(0,\infty )$, it follows that the mapping $f: s\mapsto \int_0^s\lambda (t;y)-\lambda (t;x)dt$ is continuous.
Moreover, noting that  $f(0)=f(\infty) =0$,
  we infer that $A$ is an open set, i.e.,  $A=\cup_i(a_i,b_i)$, where
$a_i,b_i\not\in A$.  
By Lemma \ref{counter example 2}, $\lambda(x)$ is a step function on $(a_i,b_i )$.
Moreover,
we claim that
$$b_i<\infty$$
for every $i$.
Otherwise, $(a_i,b_{i})=(a_i,\infty)$.
By 
 $\int_0^\infty \lambda (t;y)-\lambda (t; x)dt=0  $  and $\int_0^s\lambda (t;y)-\lambda (t; x)dt>0  $, $s\in (a_i ,\infty )$,
we obtain that   $\lambda (x)>0$ on $(0,\infty)$.
Since $x\in L_1(\cM,\tau)$ (therefore, $\tau$-compact), it follows that $\lambda (x) $ decreases to $0$ on $(a,\infty)$.
By Lemma~\ref{counter example 3}, we obtain that $x$ is not an extreme point.
Hence, $(a_i, b_i)$ is finite for any $i$.
Now, the statement of the theorem follows by the same argument in \cite[p.20--24]{DHS}.
Indeed, the case when $\lambda(x) $ takes only two   values or more  on $(a_i,b_i)$ can be inferred by Lemma \ref{counter example 3} directly.

 ``$\Leftarrow$'' The proof of this part is also similar with that in \cite{DHS}.
 However, there are some technical details which are somewhat different from that in \cite{DHS}.
 We provide the full proof below and technical lemmas in Appendix \ref{S:A}.

Let $y\in L_1(\cM,\tau)$.
Let  $x,x_1,x_2\in \Omega (y)$ with $x=\frac{x_1+x_2}{2}$.
 Assume that $x$ satisfies that  for   every $t\in (0,\infty )$, one of the followings holds:
  \begin{enumerate}
  \item[(i).]   $\lambda (t;x)=  \lambda(t; y)$;
  \item [(ii).]
   $\lambda(t;x) \ne \lambda(t;y)$ with  the spectral projection $ E^x \{\lambda (t;x)  \} $ being an atom in $\cM$ and $$\int_{ \{s;\lambda (s;x)=\lambda (t;x)\}}  \lambda (s;y)ds    =  \lambda(t ;x)     \tau(E^x (\{\lambda (t;x)\}))  .  $$
\end{enumerate}
and
one of the following options holds:
       \begin{enumerate}
  \item[(i).]   $\lambda (t;-x)=  \lambda(t; -y)$;
  \item [(ii).]
   $\lambda(t;-x) \ne \lambda(t;y)$ with  the spectral projection $ E^{-x} \{\lambda (t;-x)  \} $ being an atom in $\cM$ and $$\int_{ \{s;\lambda (s;-x)=\lambda (t;-x)\}}  \lambda (s;-y)ds    =  \lambda(t ;-x)     \tau(E^{-x} (\{\lambda (t;-x)\}))  . $$
\end{enumerate}
We claim that   $x_1=x_2=x$, i.e., $x$ is an extreme point of $\Omega (y)$.

For any $t$ such that $\lambda(t;y) \ne \lambda (t;x)$, we denote $[t_1,t_2) =\{s : \lambda (s;x) =\lambda (t;x )\}$, $t_1<t_2$.
In particular, we have
$$\int_0^{t_1}\lambda (s;y)ds =\int_0^{t_1}\lambda (s;x)ds $$
and
$$\int_0^{t_2}\lambda (s;y)ds =\int_0^{t_2 }\lambda (s;x)ds .$$
Since   $x_1,x_2\in \Omega(y)$ and
 \begin{align*}
 \int_0^{t_1} \lambda (s;2x_+)ds & \quad = \int_0^{t_1} \lambda (s;(x_1+x_2)_+)ds \\
 &\stackrel{Prop. \ref{2.5}}{\le}   \int_0^{t_1} \lambda (s;(x_1)_+)ds+\int_0^{t_1} \lambda (s;(x_2)_+)ds\\
 & \quad \le 2 \int_0^{t_1}\lambda (s;y_+)ds,
 \end{align*}
it follows that
\begin{align}\label{***}  \int_0^{t_1} \lambda (s;x_+)ds =   \int_0^{t_1} \lambda (s;(x_1)_+)ds= \int_0^{t_1} \lambda (s;(x_2)_+)ds = \int_0^{t_1}\lambda (s;y_+)ds.  \end{align}
The same argument with $t_1$ replaced with $t_2$ yields that
\begin{align}\label{****}  \int_0^{t_2} \lambda (s;x_+)ds =   \int_0^{t_2} \lambda (s;(x_1)_+)ds= \int_0^{t_2} \lambda (s;(x_2)_+)ds = \int_0^{t_2} \lambda (s;y_+)ds.  \end{align}
Let $e_1 := E^{x}(\lambda (t_1 ;x_+),\infty)$ and  $e_2:= E^{x}[ \lambda ( t_1 ;x_+),\infty)  $.
In particular, $e_2-e_1 =E^x\{\lambda (t_1;x_+)\}$.
Observe that $\tau(e_1)=t_1$ and $\tau(e_2)=t_2$ (due to the assumption that $[t_1,t_2)=\{s : \lambda (s;x_+) =\lambda (t;x_+ )\} $).
By Proposition \ref{HN1} and the definition of spectral scales $\lambda (x)$, we have
\begin{align*}
 2 \int_0^{t_1} \lambda (s;x_+)ds &\quad = \tau(2x_+ e_1) =\tau((x_1+x_2)_+e_1) = \tau(e_1E^x(0,\infty ) (x_1+x_2) E^x(0,\infty )    e_1  )\\
& \quad \le  \tau(e_1E^x(0,\infty ) ((x_1)_++(x_2)_+) E^x(0,\infty )    e_1  )\\
&\quad \le  \tau(e_1 (x_1)_+    e_1  )+ \tau(e_1 (x_2)_+    e_1  )\\
  &\stackrel{Prop. \ref{HN1}}{\le}   \int_0^{t_1} \lambda (s;(x_1)_+)ds+\int_0^{t_1} \lambda (s;(x_2)_+)ds \stackrel{\eqref{***}}{=}2\int_0^{t_1}  \lambda (s;x_+)ds .
 \end{align*}
 We obtain that $\tau(x_1e_1)=\int_0^{t_1} \lambda (s;x_1)ds
=\int_0^{t_1} \lambda (s;x_2)ds = \tau(x_2e_1 )   $.
By Corollary~\ref{corcons},
 we have
 $$ E^{x_1}(\lambda (t_1;x_1),\infty ) \le  e_1  \le E^{x_1}[\lambda (t_1;x_1),\infty )$$ and
 $$ E^{x_2}(\lambda (t_1 ;x_2),\infty ) \le  e_1  \le E^{x_2}[\lambda (t_1;x_2),\infty ) .$$
 Similar argument   with $t_1$ replaced with $t_2$ yields that
  $$ E^{x_1}(\lambda (t_2;x_1),\infty ) \le  e_2 \le E^{x_1}[\lambda (t_2;x_1),\infty )$$
  and
   $$ E^{x_2}(\lambda (t_2;x_2),\infty ) \le  e_2  \le E^{x_2}[\lambda (t_2 ;x_2),\infty ) .$$
 In particular, $e_1=E^{x_1}(\lambda (t_1;x_1),\infty )  + q$ for some subprojection $q$ of $E^{x_1}\{\lambda (t_1;x_1)\}$.
 Since $q$ commutes with  $E^{x_1}\{\lambda (t_1;x_1)\}$, it follows that $q $ commutes with any spectral projection of $x_1$. Hence, $e_1$ commutes with $x_1$.
   The same argument implies that both
   $e_1$ and $e_2$ commute with $x_1$ and with  $x_2$. 
Moreover,
 the atom $  e:= e_2-e_1\in \cP(\cM)$ satisfies that
 $$ e   \le E^{x_1} [\lambda (t_2;x_1), \lambda (t_1;x_1)]
  \mbox{ and }
   e\le  E^{x_2} [\lambda (t_2;x_2),\lambda (t_1;x_2) ] . $$
By the spectral theorem, 
 $ \lambda (x_1e_1) =\lambda (x_1)$  and $ \lambda (x_2e_1) =\lambda (x_2)$ on $(0,t_1)$ (see  \eqref{2.2}).
 On the other hand,
 $$\lambda (t; x_1e ) = \lambda (t; x_1e_2 e ) \stackrel{\eqref{2.3}}{=} \lambda (t+t_1; x_1e_2) \stackrel{\eqref{2.2}}{=}\lambda (t+t_1; x_1)$$
 and $$\lambda (t; x_2e ) = \lambda (t; x_2e_2 e ) \stackrel{\eqref{2.3}}{=}  \lambda (t+t_1; x_2e_2 ) \stackrel{\eqref{2.2}}{=} \lambda (t+t_1; x_2) $$
 for all $t\in [0,t_2-t_1)$.
 Since $e $ is an atom, it follows that $\lambda _1 :=\lambda (t; x_1e ) = \lambda (t+t_1; x_1)$
 and $\lambda _2 :=\lambda (t; x_2e ) = \lambda (t+t_1; x_2)$ for every $t\in [0,t_2-t_1)$.
Combining \eqref{***} and \eqref{****}, we have
   \begin{align}\label{t12} \int_{t_1}^{t_2} \lambda (s;x)ds& =    \int_{t_1}^{t_2} \lambda (s;x_1)ds =\lambda _1 (t_2-t_1)\nonumber\\
   &=  \int_{t_1}^{t_2} \lambda (s;x_2)ds =  \lambda _2  (t_2-t_1)= \int_{t_1}^{t_2} \lambda (s;y)ds.
   \end{align}
Hence,  $\lambda_1 =\lambda _2 = \lambda (t_1 ;x) $.

Let
$A=\{ s: \lambda (s;x)=\lambda (t;x) \mbox{ for some $t$ such that }~\lambda (t;x)\ne \lambda (t;y)\}$.
Note that \eqref{t12} holds any interval $[t_1,t_2):=\left\{ s: \lambda (s;x)=\lambda (t;x) \right\}$ for some $t$ such that $\lambda (t;x)\ne \lambda (t;y)$.
For any $t\in [0,\infty )\setminus A$, we have
\begin{align*}
\int_0^t \lambda (s;2x_+)ds \stackrel{Prop. \ref{2.5}}\le \int_0^t \lambda (s;(x_1)_+ )ds  +\int_0^t \lambda (s;(x_2)_+ )ds
&\le \int_0^t \lambda (s;2 y_+)ds \\
&\stackrel{(i)}{=}\int_0^t \lambda (s;2 x_+)ds .
\end{align*}
Hence,
\begin{align*}
\int_0^t \lambda (s;2x_+)\chi_{[0,\infty)\setminus A}ds & \le \int_0^t \lambda (s;(x_1)_+ )\chi_{[0,\infty)\setminus A}ds  +\int_0^t \lambda (s;(x_2)_+ )\chi_{[0,\infty)\setminus A}ds
 \\
&\le \int_0^t \lambda (s;2 y_+)\chi_{[0,\infty)\setminus A}ds=\int_0^t \lambda (s;2 x_+)\chi_{[0,\infty)\setminus A}ds ,
\end{align*}
which
implies that $\lambda (y)\chi_{[0,\infty)\setminus A}=\lambda (x)\chi_{[0,\infty)\setminus A} =\lambda (x_1)\chi_{[0,\infty)\setminus A}=\lambda(x_2)\chi_{[0,\infty)\setminus A} $ a.e..
Hence, by right-continuity,  $\lambda (x_+)=\lambda ((x_1)_+)=\lambda ((x_2)_+)$.

The same argument shows that
$\lambda ((x_1)_-)=\lambda ((x_2)_-)=\lambda (x_-)$.
By Proposition~\ref{noncomm KPS lemma}, we obtain that
 $x_1=x_2=x$.
That is, $x$ is an extreme point of $\Omega (y)$.
\end{proof}

\appendix
 \section{Technical results}\label{S:A}
Throughout this appendix, we always assume that   $\cM$ is  a semifinite von Neumann algebra equipped with a semifinite infinite faithful normal trace $\tau$.
Some of the results in this section are well-known for positive operators (see e.g. \cite{DPS,HSZ,CKSa,CS}).
Main ideas used in this section come  from \cite{DHS}.
However,
  dealing with technical obstacles in the infinite setting requires additional care.

The following is a noncommutative   analogue of Ryff's Proposition stated in \cite{Ryff}.
The case for finite von Neumann algebras can be found in \cite{DHS}.
One should note that $\lambda (t;x)=0$ for some $t<\infty$ does  not implies that $E^x\{0\}$ is not trivial.
\begin{lem}\label{counter example 1} Let $y  \in L_1(\cM, \tau)_{h}.$ If $x \in\Omega(y)=\{x\in L_1(\mathcal{M},\tau)_h:\ x\prec y \}$ satisfies that
\begin{enumerate}[{\rm (i)}]
  \item $0\le  \lambda (s_{i+1};x_+ )< \lambda(s_i;x_+)$ (or $0\le  \lambda (s_{i+1};x_-)< \lambda(s_i;x_-)$) for some $0<s_i<s_{i+1}<\infty,$ $i=1,2,3$,
  \item $ \int_0^{s_1}\lambda (t;x_+  )dt+\lambda (s_1;x_+ )(s-s_1) \le \int_0^{s}\lambda (t; y_+)dt ,$
      (or $ \int_0^{s_1}\lambda (t;x_-  )dt+\lambda (s_1;x_- )(s-s_1) \le \int_0^{s}\lambda (t; y_-)dt $) for all $s\in[s_1,s_4],$
\end{enumerate}
then $x \not\in\emph{extr}(\Omega(y)).$
\end{lem}

\begin{proof}
For the sake of convenience, we denote $a_i =\lambda(s_i;x_+ )$, $i=1,2,3,4$.
Note that $a_1 >a_2 >a_3>a_4\ge 0$.

Let $p_1 := E^{x_+ } [\frac{a_2+a_3}{2}, \frac{a_1+a_2}{2})$ and $p_2:= E^{x _+ } [\frac{a_3+a_4}{2},\frac{a_2+a_3}{2})$.
We denote $T_1= \tau(p_1 )$ and $T_2 = \tau(p_2) $.
Observe that  $T_1,T_2<\infty $ and $p_1p_2 =0 $.
Set
$$ u : =  p_1 -  \frac{T_1}{T_2} p_2 . $$
It is clear that $\tau(u) =0$.
Assume that $\delta>0$ such that $$\delta < \min\{ \frac{a_1-a_2}{2}, \frac{(a_3-a_4)T_2 }{2 T_1} , \frac{ (2a_1-a_2-a_3)T_2 }{2T_1} , \frac{a_2+a_3-2a_4}{2}\} .$$
Let
$$x_1:=x + \delta u  \mbox{ and }x_2:=x- \delta u  .$$
By the spectral theorem, $\lambda(s;x_1)=\lambda(s;x_2)=\lambda (s;x )$ for
$s\not\in[s_1,s_4]$ and $\lambda(s;x_{i})\leqslant \lambda (s_1;x  )$ for
$s\in[s_1,s_4]$, $i=1,2$.
We assert that $x_{1},x_2\prec y  $.


Note that for $i=1,2$, we have
\begin{eqnarray*}
\int_0^{\infty}\lambda(t;(x_{i})_+)dt -
\int_0^\infty \lambda (t;(x_i)_-)dt =
 \tau (x_{i})    =\tau( x \pm \delta u)=\tau(x )=\tau(y ).
\end{eqnarray*}
Since $\lambda(s;x_{i})=\lambda (s;x)$ for
$s\not\in[s_1,s_4]$ and $i=1,2$,  it  follows that
$$\int_{s_1}^{s_4}\lambda(t;(x_i)_+ )dt-\int_{s_1}^{s_4}\lambda (t;x_+) dt=\int_0^{\infty }\lambda(t;(x_i)_+)dt-\int_0^{\infty }\lambda (t;x_+) dt=0,~i=1,2.$$
Hence, for $i=1,2$, we have
$$\int_0^s\lambda(t;(x_{i})_+)dt=\int_0^s \lambda (t;x_+)dt\leqslant\int_0^s \lambda (t;y_+  )dt,\quad s\not\in[s_1,s_4],$$
where the last inequality follows from the assumption that  $x\in\Omega(y).$
On the other hand, since $\lambda (x)$ is decreasing, it follows that
$$\int_0^{s}\lambda(t;(x _{i})_+)dt\leqslant\int_0^{s_1}\lambda (t;x )dt+ \lambda (s_1;x_+ )(s-s_1)\stackrel{(ii)}{\leqslant }\int_0^{s}\lambda (t; y _+ )dt,\quad s\in[s_1,s_4].$$
Hence, $x_1,x_2\in\Omega(y )$ and $x =\frac12(x_1+x _2).$
That is, $x \not\in\emph{extr}(\Omega(y))$.
\end{proof}

The following lemma extends \cite[Lemma 3.2]{DHS}.
\begin{lem}\label{counter example 2} Let $y\in L_1(\cM,\tau)_h.$
If $x  \in\emph{extr}(\Omega(y))$ and $\int_0^s \lambda (t;y_\pm )dt>\int_0^s\lambda (t;x_\pm)dt$ for all $s\in(t_1,t_2),$ then for any $s\in (t_1,t_2)$,
$\lambda (x_\pm )$ is a constant on $[s,s+\varepsilon)$ for some $\varepsilon>0$.
In particular, $\lambda (x_\pm) $ is a step function on $[t_1,t_2)$.
Here, $0\le t_1<t_2\le \infty $.
\end{lem}
\begin{proof}
We only consider the case for $x_+$ and $y_+$.
The case for $x_-$ and $y_-$ follows from the same argument.
For the sake of simplicity, we may assume that $x,y$ are positive.
Note that if $\lambda (t_1;x_+)=0$, then $\lambda(x_+)$ is a constant function on $(t_1,t_2)$.
Hence, we may assume that $\lambda (t_1; x_+)>0$.
If $\lambda (t_2;x_+)=0$, by the right-continuity, there exists a $t_2'\in (t_1,t_2]$ such that $\lambda (t_2'; x_+ ) =0 $ and $\lambda (t_2' -\varepsilon ; x_+)>0$ for any $\varepsilon>0$.
Hence,
without loss of generality,
we may assume that $\lambda (t_2-\varepsilon ;x_+)>0$  for any $\varepsilon>0$.

Assume by contradiction that there exists $s_1\in (t_1,t_2)$   such that  
 $(s_1,s_1+\varepsilon)$ is  not a  constancy interval  of $\lambda (x )$ for any  $\varepsilon>0.$
Without loss of generality, we assume that $\lambda(s_1+\varepsilon;x)>0$. Since $\lambda(x)$ is right-continuous, one can choose $s_2\in (s_1, s_1+\varepsilon)$ such that
$$\int_0^s \lambda (t; y) - \lambda (t;x)dt \ge 0,~s\in [s_1,s_2), $$
and
$$\frac{\int_0^{s_1}\lambda (t;y)-\lambda (t;x)dt}{\lambda (s_1; x)}\ge s_2-s_1.$$
Hence, for any $s\in [s_1,s_2]$, we have
$$ \int_0^{s} \lambda (t;y)dt \ge   \int_0^{s_1}\lambda (t; y)dt \ge   \int_0^{s_1} \lambda (t;x)dt +  \lambda (s_1;x) ( s_2-s_1) .  $$
Since $y$ and $x$ satisfy the assumptions in  Lemma~\ref{counter example 1}, it follows that $x\not\in\text{extr}(\Omega(Y)).$
\end{proof}

The following lemma covers cases 1 and 2 in the proof of  \cite[Theorem 1.1]{DHS}.
\begin{lem}\label{counter example 3} Let $y  \in L_1(\cM, \tau)_{h}.$ If $x \in\Omega(y)$ satisfies that
$$0\le \lambda (s_3;x_+)<\lambda (s_2;x_+)<\lambda (s_1;x_+)<\lambda (s_1-;x_+)$$
(or $0\le \lambda (s_3;x_-)<\lambda (s_2;x_-)<\lambda (s_1;x_-)<\lambda (s_1-;x_-)$) for some $0\le s_i<s_{i+1}<\infty,$ $i=1,2$ (for convenience, we define $\lambda (s_1-;x_+)=\infty $ when $s_1=0$), and $\int_0^s \lambda(x_+)<\int_0^s \lambda(y_+)$ (or $\int_0^s \lambda(x_-)<\int_0^s \lambda(y_-)$) on $(s_1,s_3)$,
then $x \not\in\emph{extr}(\Omega(y)).$
\end{lem}

\begin{proof}
Assume by contradiction that $x\in \emph{extr}(\Omega (y))$.
For the sake of convenience, we denote $a_i =\lambda(s_i;x_+ )$, $i=1,2$.
By Lemma \ref{counter example 2}, for every $t\in (s_1,s_4)$, $\lambda(x_+)$ is a constant on $[t,t+\varepsilon)$ for some sufficiently small $\varepsilon>0$.
Now, let's consider $s_1$.
Assume that $\lambda (x_+)$ is not constancy on $[s_1,s_1+\varepsilon)$ for any $\varepsilon>0$.
Then, by right-continuity, there
exists a sequence of positive numbers $\varepsilon_n$
decreasing to $0$ such that
$\{\lambda(s_1+\varepsilon_n;x_+)\} $
strictly increases to $\lambda (s_1;x_+)$.
Then, we may replace
$s_1$, $s_2$ and $s_3$ with $s_1+\varepsilon_n$,
$s_1+\varepsilon_{n-1}$ and $s_1+\varepsilon_{n-2}$, respectively.
By Lemma \ref{counter example 2}, we obtain that
 $\lambda (x_+)$ are constancy on
 $[s_1 , s_1  +\delta_1 )$,
$[s_1+\delta_1 , s_1 +\delta_2 )$ and $[s_1+\delta_2 , s_2 + \delta_3 )$
for some $\delta_1,\delta_2,\delta_3 >0$.
Hence, without loss of generality,   we may assume that $\lambda (s;x_+)=a_i$ on $[s_i,s_{i+1})$, $i=1,2$.

Let $p_1 := E^{x_+ } \{a_1\}$ and $p_2:= E^{x _+ } \{a_2\}$.
We denote $T_1= \tau(p_1 )$ and $T_2 = \tau(p_2) $.
Observe that  $T_1,T_2<\infty $ and $p_1p_2 =0 $.
Set
$$ u : =  p_1 -  \frac{T_1}{T_2} p_2 . $$
It is clear that $\tau(u) =0$.
Let $C:= \frac{1}{s_2-s_1}\left(\int_{0}^{s_2} \lambda (y) -\int_{0}^{s_1} \lambda (x) \right)$.
Clearly, $C>\frac{1}{s_2-s_1}\int_{s_1}^{s_2} \lambda (x)=a_1$.

Assume that $\delta>0$ such that $$\delta < \min\left\{C-a_1, (a_-)-a_1, \frac{a_1-a_2}{1+\frac{T_1}{T_2}}, \frac{ (a_2-a_3)T_2 }{T_1} \right\} .$$
Let
$$x_1:=x + \delta u  \mbox{ and }x_2:=x- \delta u  .$$
By the spectral theorem, $\lambda(s;x_1)=\lambda(s;x_2)=\lambda (s;x )$ for
$s\not\in[s_1,s_3]$ and $\lambda(s;x_{i})\leqslant \lambda (s_1;x  )$ for
$s\in[s_1,s_3]$, $i=1,2$.
We assert that $x_{1},x_2\prec y_+  $.


Note that
\begin{eqnarray*}
\int_0^{\infty}\lambda(t;(x_{1})_+)dt -
\int_0^\infty \lambda (t;(x_1)_-)dt =
 \tau (x_{1})    =\tau( x+ \delta u)=\tau(x )=\tau(y ).
\end{eqnarray*}
Since $\lambda(s;x_{i})=\lambda (s;x)$ for
$s\not\in[s_1,s_3]$ and $i=1,2$,  it  follows that
$$\int_{s_1}^{s_3}\lambda(t;(x_i)_+ )dt-\int_{s_1}^{s_3}\lambda (t;x_+) dt=\tau(x_i)-\tau(x)=0,~i=1,2.$$
Hence, for $i=1,2$, we have
$$\int_0^s\lambda(t;(x_{i})_+)dt=\int_0^s \lambda (t;x_+)dt\leqslant\int_0^s \lambda (t;y_+  )dt,\quad s\not\in[s_1,s_3],$$
where the last inequality follows from the assumption that  $x\in\Omega(y).$
Since $s\mapsto \int_0^s \lambda (t;y)dt$ is a concave function,
$$\int_0^{s_1} \lambda (t;y_+)dt\ge \int_0^{s_1} \lambda (t;(x_1)_+)dt,$$
$$\int_0^{s_2} \lambda (t;y_+)dt\ge \int_0^{s_2} \lambda (t;(x_1)_+)dt $$
and $\lambda ((x_1)_+)$ is constancy on $[s_1,s_2)$,
it follows that $\int_0^s\lambda (t;x_1)dt\le \int_0^s \lambda (t;y)dt$ for any  $s\in [s_1,s_2)$.
The same argument show that
$\int_0^s\lambda (t;x_1)dt\le \int_0^s \lambda (t;y)dt$ on for any $s\in[s_2,s_3)$.
Hence, $x_1\in\Omega(y )$ (similarly, $x_2\in \Omega (y)$) and $x =\frac12(x_1+x _2).$
That is, $x \not\in\emph{extr}(\Omega(y))$.
\end{proof}

Let $x \in L_1(\cM,\tau)_+$.
Denote by $\cN_x$ the abelian von Neumann (reduced) algebra generated by all spectral projections $E^{x}(t,\infty)$ of $x $, $t>0$.
We define $\cN:= \cN_x \oplus \cM_{{\bf 1} -s(x)}$.
In particular, $x\in L_1(\cN,\tau)_h$.
 Since $x$ is $\tau$-compact,
 it follows that $\tau$ is again semifinite on $\cN$.
Recall that (see \cite[Proposition 1.1]{HN1})
\begin{align}\label{Hineq}\int_0^s   \lambda (t;x)dt =\sup \{\tau(x a);a\in \cM, 0\le a\le 1, \tau(a)=s\}.
\end{align}
Assume that $\tau(xe) = \int_0^s  \lambda (t;x)dt$ for a projection $e\in \cM$ with $\tau(e)=  s $.
Let $E_\cN$ be the conditional expectation from $L_1(\cM,\tau)$ onto $ L_1(\cN,\tau)$ \cite{U} (see also \cite[Proposition 2.1]{DDPS}).
In particular,  $E _\cN(e)\le {\bf 1} $ and $\tau(E_\cN(e))=\tau(e)=s$.
Moreover, \begin{align}\label{ecnex}
\tau(E_\cN(e)x) \stackrel{\mbox{\cite[Prop. 2.1]{DDPS}}}{=}\tau(ex) =\int_0^s \lambda (t;x)dt.
\end{align}
We note that $E_\cN(f)$, $f\in \cP(
\cM)$,  is not necessarily a projection \cite{Sukochev}.
The proof of the following proposition is similar with \cite[Proposition 3.3]{DHS}.
We provide a proof for completeness.
\begin{prop}Under the above assumptions on $e$, we have
\begin{align}\label{te}
E^{x}(\lambda (s;x),\infty) \le E_\cN(e) \le E^{x}[\lambda (s;x),\infty)
.\end{align}
\end{prop}
\begin{proof}
We present the proof for  the first inequality and a similar argument yields that $E_\cN(e) \le E^x [\lambda (s;x),\infty)$.

Without loss of generality, we may assume that $E^{x}(\lambda (s;x),\infty)\ne 0$.
Since $\cN_x$ is a commutative algebra, $\cM_{{\bf 1}-s(x)}\perp \cN_x$ and $0\le E_\cN(e)\le {\bf 1} $, it follows that
$$E_\cN(e) E^x (\lambda (s;x), \infty)= E^x (\lambda (s;x), \infty)^{1/2} E_\cN(e)  E^x (\lambda (s;x), \infty) ^{1/2} \le  E^x (\lambda (s;x), \infty).  $$
 If $E_\cN(e)E^x(\lambda(s;x),\infty) = E^x(\lambda(s;x),\infty)$, then \begin{align}\label{ENE>LAMBDA}
 E_\cN(e) \ge E_\cN(e)^{1/2} E^x(\lambda(s;x),\infty)E_\cN(e)^{1/2} &= E_\cN(e) E^x(\lambda(s;x),\infty) \nonumber\\
 &= E^x(\lambda(s;x),\infty),
 \end{align}
 which proves the first inequality of \eqref{te}.

 Now, we  assume that  $$E_\cN(e) E^x(\lambda (s;x),\infty) < E^x(\lambda (s;x),\infty) .$$
 This implies that
 \begin{align}\label{EcNe1}
  \tau( E_\cN(e) E^x(\lambda (s;x),\infty))< \tau(E^x(\lambda (s;x),\infty) ).
  \end{align}
  Since $\tau(E^x(\lambda (s;x), \infty ))\le s$,
  it follows that $\lambda (t;E_\cN(e) E^x(\lambda (s;x),\infty)  ) =0$ when $t>s$.
  Hence,
   \begin{align*}
   \int_0^s  \lambda (t;E_\cN(e) E^x(\lambda (s;x),\infty) ) dt &=\tau( E_\cN(e) E^x(\lambda (s;x),\infty))\\
   &\stackrel{\eqref{EcNe1}}{<} \tau(E^x(\lambda (s;x),\infty) )=  \int_0^s  \lambda (t;  E^x(\lambda (s;x),\infty) )dt . \end{align*}
   In particular,
    \begin{align}\label{1-}
    \int_0^s 1- \lambda (t;E_\cN(e) E^x(\lambda (s;x),\infty) ) dt >0   .
    \end{align}
Assume that $\tau(E_\cN(e) E^x(\lambda (s;x),\infty))=a_1 \ge  0$ and $\tau(E_\cN(e) E^x(-\infty, \lambda (s;x)] )=a_2  \ge  0 $.
Observe that $a_1+a_2 =s=\tau(E_\cN(e))$.
Moreover,  by \eqref{disx}, for every $0 <t< \tau(E^x( \lambda (s;x) ,\infty ))$, we have  $\lambda (t;x) > \lambda (s;x)$. Hence,
\begin{align*}
  \int_0^s  \lambda (t;x)dt &=  \int_0^s  \lambda (t;x)  \lambda (t;E_\cN(e) E^x(\lambda (s;x),\infty)  )dt \\
&\qquad +  \int_0^s  \lambda (t;x)  (1- \lambda (t; E_\cN(e) E^x(\lambda (s;x),\infty)  )dt \\
 & \stackrel{\eqref{1-}}{>}   \int_0^s  \lambda (t;x)  \lambda (t;E_\cN(e) E^x(\lambda (s;x),\infty)  )dt\\
 &\qquad +   \lambda (s;x)  \int_0^s  (1- \lambda (t; E_\cN(e) E^x(\lambda (s;x),\infty)  )dt.
 \end{align*}
Recall that  $\int_0^s  \lambda (t;E_\cN(e) E^x(\lambda (s;x),\infty) ) dt =\tau( E_\cN(e) E^x(\lambda (s;x),\infty))$. The above inequality implies that
\begin{align*} &\quad  \int_0^s  \lambda (t;x)dt \\
& >   \int_0^\infty   \lambda (t;x)  \lambda (t;E_\cN(e) E^x(\lambda (s;x),\infty)  )dt +     \lambda (s;x) \Big(s - \tau( E_\cN(e) E^x(\lambda (s;x),\infty)  ) \Big) \\
    & =  \int_0^\infty   \lambda \Big(t;x  \Big)  \lambda \Big(t;E_\cN(e) E^x(\lambda (s;x),\infty)  \Big)dt+     \lambda (s;x)   a_2 .
    \end{align*}
Since $ \lambda (s;x) a_2 = \tau( \lambda (s;x)E_\cN(e) E^x(-\infty, \lambda (s;x)] ) \ge  \tau( E_\cN(e) E^x(-\infty, \lambda (s;x)] x)$, it follows that
\begin{align*}  \int_0^s  \lambda (t;x)dt  &\stackrel{\eqref{2.4}}{> }    \tau \Big( xE^x(\lambda (s;x),\infty)  E_\cN(e)   \Big)   +         \tau\Big(  x E^{x}(-\infty,\lambda (s;x)]E_\cN(e)     \Big)\\
&=  \tau(xE_\cN(e)   ),
 \end{align*}
which is a contradiction with \eqref{ecnex}.
Hence, the equality $E_\cN(e)E^x (\lambda (s;x),\infty)= E^x (\lambda (s;x),\infty)$ holds, and therefore, by \eqref{ENE>LAMBDA}, we have $E_\cN(e) \ge E^x (\lambda (s;x),\infty)$.
\end{proof}

\begin{lem}\label{reduction to commutative} Let $x\in L_1(\mathcal{M},\tau)_+$. Let $0<s< \tau(\mathbf 1)=\infty $ and let $a $ be in the unit ball of $ \mathcal{M}_+ $  such that $\tau(a)=s$  and $\tau(xa)=\int_0^s\lambda(t;x)dt.$
If $\lambda(x)$ is not a constant in any left neighborhood of $s,$ then $a=E^ x(\lambda(s;x),\infty).$
\end{lem}
\begin{proof} 
Let $\cN= \cN_x\oplus \cM_{{\bf 1}-s(x)}$,
where
  $\mathcal{N}_x$ is   the  commutative weakly closed $*$-subalgebra of $\mathcal{M}$ generated by the spectral projections of $x$. Clearly, the restriction of $\tau$ to $\mathcal{N}$ is semifinite.
 There exists a conditional   expectation  $E_\cN$ from $L_1(\mathcal{M},\tau)$ to $ L_1(\mathcal{N},\tau)$~\cite[Proposition 2.1]{DDPS}.
In particular, for any $z\in L_1(\cM,\tau)+\cM$, we have $ E_\cN (z) \prec\prec z  $ (see e.g. \cite[Proposition 2.1 (g)]{DDPS}).
Moreover, for every $z\in \cM$ and $y\in L_1(\cN,\tau)$, we have
\begin{align}\label{exp}
\tau(yz)=\tau(E _\cN (yz))=\tau(y E_\cN( z) ).
\end{align}

Since $a$ is positive, it follows that $\lambda(a)=\mu(a)$ (see  \cite{HN1}) and, therefore,
$$\int_0^r\lambda(t;E_\cN (a))dt\leqslant\int_0^r\lambda(t;a )dt
$$
for all $r\in(0,\infty)$ and, 
 $$\tau(E_\cN  (a ))\stackrel{\mbox{\cite[Prop. 2.1. (h)]{DDPS}}}{=} \tau(a) .$$
Moreover, since $E_\cN$ is a contraction on $\cM$ and $\left\|a\right\|_\infty \le 1$, it follows that $\lambda(E _\cN (a ))\leqslant1$ \cite[Proposition 2.1 (g)]{DDPS}.
We set  
$y:=(x-\lambda(s;x) )_+ .$
Note that
$$x\leqslant\lambda(s;x)+y .$$
Therefore,
\begin{equation}\label{trace}
\tau(xa )\leqslant\tau(\lambda(s;x)a)+\tau(ya)=s\cdot \lambda(s;x) +\tau(ya).
\end{equation}
Since $\lambda(x)$ is a decreasing function, we have
  \begin{align}\label{yfinite}\lambda(t;y)=\left\{
                 \begin{array}{ll}
                   \lambda(t;x)-\lambda(s;x), & \hbox{if}\ 0<t<s; \\
                   0, & \hbox{if}\ s\leqslant t<\infty .
                 \end{array}
               \right.
               \end{align}
It follows from $\tau(a)=s$  and $\tau(xa)=\int_0^s\lambda(t;x)dt$ that
\begin{align*}
  \int_0^s\lambda(t;y)dt &\stackrel{\eqref{yfinite}}{=} \int_0^s\lambda(t;x)dt-\int_0^s\lambda( s  ;x) dt  \\
  &~= \tau(xa)- \lambda(s;x) \tau( a )  \\
  &\stackrel{\eqref{trace}}{\leqslant} \tau(ya) \\
  &\stackrel{\eqref{exp}}{=}  \tau(yE _\cN (a)) \\
  &\stackrel{\eqref{2.4}}{\leqslant} \int_0^{\infty }\lambda(t;y)\lambda(t;E_\cN (a ))dt \\
  &\stackrel{\eqref{yfinite}}{=} \int_0^{s}\lambda(t;y)\lambda(t;E _\cN(a))dt.
\end{align*}
Thus,
$$\int_0^{s}\lambda(t;y)(1-\lambda(t;E_\cN  (a)))dt\leqslant0.$$
Since $y$ is positive and $\lambda(E_\cN(a))\leqslant1,$ we conclude that $\lambda(t;y)(1-\lambda(t;E_\cN(a)))\geqslant0.$
Hence,  $\lambda(t;y)(1-\lambda(t;E_\cN(a)))=0$ for all $t\in(0,s).$
Recall that $\lambda(x)$ is not a constant in any left neighborhood of $s$. We obtain that  $\lambda (y)>0$ on $(0,s)$.
Recall that $E_\cN(a)\ge 0$ with $\tau(E_\cN(a) )=\tau(a) =s$.
 We obtain that  $\lambda (E_\cN(a))=1$ on $(0,s)$ and $\lambda (E_\cN(a))=0$ on $[s,\infty)$.
This implies that
$E_\cN(a)$ is a projection in $\cN$.
Hence, $E_\cN(a)= E_\cN(a) E_\cN(a) =E_\cN(a\cdot E_\cN(a))$ and $E_\cN(a( {\bf 1}-E_\cN(a)))=0$.
It follows that $$\tau(a^{1/2} ({\bf 1}-E_\cN(a)) a^{1/2} ) =\tau(a ({\bf 1}-E_\cN (a))  )=\tau(E_\cN(a( {\bf 1} -E_\cN(a)))  ) =0  .$$
Therefore, $a^{1/2}({\bf 1} -E_\cN(a))a^{1/2} =0$ and $a^{1/2} =E_\cN (a) a^{1/2} $.
By the assumption that  $a\le {\bf 1}$, we have
$$ E_\cN(a)= E_\cN(a)^2   \ge   E_\cN(a)a E_\cN(a)   =   a . $$
Recall that $\tau(E_\cN(a))=\tau(a)$.
Hence, $\tau(E_\cN(a)-a )=0$.
Due to the faithfulness of the trace $\tau$, we obtain that $a = E_\cN(a )  \in \cP(\cN)$.
Since  $\mathcal{N}$ commutes with  $x$,
 it follows that $a= E^{x}\{B\}\oplus q$ for some Borel set $B\subset [0,1]$ and some projection  $p\in \cM_{{\bf 1} -s(x)}$.
 By \cite[Remark 3.3]{FK} and the assumption that $\tau(ax)=\int_0^s \lambda (t;x)dt$, we have
$$\int_0^s \lambda (t;x)dt =\tau(a x )=\tau(E^x\{B\} x )=\int_0^\infty  \chi_{B}(\lambda(t;x))dt .$$
Moreover, since  $\lambda(\cdot;x)$ is decreasing and is non-constant  in any left neighborhood of $s$, it follows  that $B=(\lambda(s;x),\infty)$. 
Since $\tau(E^x(\lambda(s;x),\infty))=s$, it follows that $a=E^x(\lambda(s;x),\infty)$.
\end{proof}

\begin{rmk}\label{posi}
Let $x\in L_1(\cM,\tau)_+$.
By \eqref{2.4},  for any $a$ in the unit ball of $\cM _+$ with $\tau(a)< s$, we have
$$\tau(xa) \le \int_0^\infty  \lambda (t;x) \lambda (t;a)dt.$$
Since $\lambda (x)$ is a non-negative decreasing function, $0\le \lambda (a)\le 1 $ and $\tau(a)=\int_0^\infty \lambda(t;a)dt <s$, it follows that
\begin{align}\label{<<<}
\tau(xa) \le \int_0^\infty  \lambda (t;x) \lambda (t;a)dt <  \int_0^s  \lambda (t;x)dt  .
\end{align}
Then, by \eqref{Hineq}, we obtain that
$$\int_0^s   \lambda (t;x)dt =\sup \{\tau(xa);a\in \cM, 0\le a\le 1, \tau(a)\le s\}.$$
Lemma \ref{reduction to commutative} together with \eqref{<<<} implies that  if  $a $ is  in the unit ball of $ \mathcal{M}_+ $  such that $\tau(a)\le  s$  and $\tau(xa)=\int_0^s\lambda(t;x)dt,$
and
$\lambda(x)$ is not a constant in any left neighborhood of $s,$ then $a=E^ x(\lambda(s;x),\infty).$
\end{rmk}

\begin{cor}\label{corcons} Let $x\in L_1(\mathcal{M},\tau)_h $. Let $0<s< \tau(\mathbf 1)=\infty  $ and let $e \in  \cP(\mathcal{M})$ be  such that $\tau(e )=s$ and $\tau(xe)=\int_0^s\lambda(t;x)dt.$
Then,
$$E^{x }(\lambda (s;x   ),\infty ) \le  e \le E^{x }[\lambda (s;x ),\infty ).$$
\end{cor}
\begin{proof}
By Lemma \ref{reduction to commutative}, it suffices to prove the case when $ \lambda (x)$ is a constant on a left neighbourhood of $s$.
Denote by $\cN:=\cN_x\oplus \cM_{{\bf 1}-s(x)}$, where $\cN_x$ is  the reduced von Neumann algebra generated by all spectral projections of $x $.
Let $E_\cN $ be the conditional expectation from $L_1(\cM,\tau)$ onto $ L_1(\cN,\tau)$.
Let
$$\lambda := \lambda (s;x) \mbox{ and  } x_1: =(x-\lambda )E^x (\lambda,\infty) .$$
Recall that $E^x[ \lambda,\infty) \ge E _\cN (e)\ge E^x(\lambda,\infty)$ (see \eqref{te}).
In particular, $E_\cN(e)E^x(\lambda,\infty) =E^x(\lambda,\infty) $.
Observing that $E^x(\lambda,\infty)$ is the support of $x_1$, we have
\begin{align*}
\tau(x_1  E^x (\lambda,\infty) e E^x (\lambda,\infty)) &~=\tau(x_1 e  )=\tau( x_1  E_\cN(  e) )\\
&~=\tau( x_1 E^x (\lambda,\infty) E_\cN( e)) \\
&~=\tau( x_1 E^x (\lambda,\infty)  ) \\
&~=\tau( (x-\lambda ) E^x (\lambda,\infty)  ) \\
&~=\tau( x E^x (\lambda,\infty)  ) -\tau( \lambda E^x (\lambda,\infty)  ) \\
&~=\int_0^{\tau(E^x (\lambda,\infty))} (\lambda (t;x E^x (\lambda,\infty) ) -\lambda )dt\\
&\stackrel{\eqref{2.2}}{=}\int_0^{\tau(E^x (\lambda,\infty))} (\lambda (t;x) -\lambda )dt\\
&~=\int_0^{\tau(E^{x_1} (0,\infty))} \lambda (t;x_1)dt.
\end{align*}
Note that $0\le E^x(\lambda,\infty   )eE^x(\lambda,\infty )\le \textbf{1}$ and $\tau(E^x(\lambda,\infty   )eE^x(\lambda,\infty )) \le \tau(E^x(\lambda,\infty   ) ) $.
Since $x_1\ge 0$, it follows from  Remark \ref{posi} that
$$E^x (\lambda,\infty) e E^x (\lambda,\infty) =E^{x_1} (0,\infty)=E^{x } (\lambda ,\infty)  .$$
That is,
$e \ge E^x (\lambda,\infty) $.
Let $e_1: = e -E^{x}(\lambda,\infty)\in \cP(\cM)$.
We have $$\tau(x e) =\tau(x ( E^{x}(\lambda,\infty)  + E^{x}(-\infty,\lambda]) e  )=\tau(x E^{x}(\lambda,\infty)  +x E^{x}(-\infty,\lambda]e_1).$$
Hence, by the  assumption that $\lambda=\lambda(s;x)$, we obtain that
\begin{align*}
\tau(x E^{x}(-\infty,\lambda]e_1)& = \int_{0}^s \lambda (t;x)dt -\int_0^{\tau(E^{x}(\lambda,\infty))} \lambda (t;x)dt\\
&=\int_{\tau(E^{x}(\lambda,\infty))}^s \lambda (t;x)dt  =\tau(\lambda e_1) .
\end{align*}
We have
$$ \tau(e_1 (\lambda  - x E^{x}(-\infty,\lambda ])e_1) =0.$$
Therefore, $ e_1 (\lambda - x E^{x}(-\infty,\lambda])e_1=0  $.
Since $\lambda - x E^{x}(-\infty,\lambda]\ge 0$, it follows that
$(\lambda - x E^{x}(-\infty,\lambda])^{1/2}e_1=0$.
Hence,  $$ \left(\int_{t <\lambda  } (\lambda -t) dE_t^x + \int_{t > \lambda  } \lambda  dE_t^x \right) e_1 =(\lambda - x E^{x}(-\infty,\lambda])e_1  = 0$$
and
\begin{align*}
&~E^x  \Big( (-\infty,\lambda)\cup (\lambda ,\infty)  \Big) \cdot e_1 \\
&= \left(\int_{t <\lambda  } \frac{1}{\lambda -t} dE_t^x + \int_{t > \lambda  } \frac1\lambda    dE_t^x \right) \left(\int_{t <\lambda  } (\lambda -t) dE_t^x + \int_{t > \lambda  } \lambda  dE_t^x \right) e_1 = 0.  
\end{align*}
This  implies that  $ e_1\le E^x\{\lambda (s;x)\} $, which completes the proof.
\end{proof}

The following proposition is similar to a well-known property of rearrangements of functions, see \cite[property $9^0$, p. 65]{KPS} and \cite[Theorem 3.5]{Hiai} and \cite{DHS}.

\begin{prop}\label{noncomm KPS lemma} Let $x,\ x_1,\ x_2\in L_1(\mathcal{M},\tau)_h$ be such that $x=(x_1+x_2)/2$ and $\lambda((x_1)_+)=\lambda((x_2)_+)=\lambda(x_+)$ and $\lambda((x_1)_-)=\lambda((x_2)_-)=\lambda(x_-)$.
Then, $x=x_1=x_2$.
\end{prop}

\begin{proof}
Fix $\theta \in (0  , \lambda (0;x_+) )  $.
Define $s$ by setting
$$s := \min \{ 0\le v\le \infty  :\lambda (v;x) \le  \theta \}. $$
If $s>0$,  then $$ \lambda(s;x) \le \theta <   \lambda(s-\varepsilon;x), ~\forall \varepsilon\in (0,s). $$
This means that $\lambda(x)$ is not constant in any left neighborhood of $s$ whenever $s>0$.

Fix a projection $$e=E^{x_1+x_2}(\lambda(s; (x_1+x_2)_+),\infty)=E^{2x_+}(\lambda(s; 2x ),\infty)=E^{x}(\lambda(s; x_+ ),\infty)\in\mathcal{M}.$$ Clearly, $\tau(e)=s$.  
By \eqref{Hineq}, we have
\begin{align*}
\tau(e(x_1+x_2))= \tau(e(x_1+x_2)_+) &~= \int_0^{s}\lambda(t;(x_1+x_2)_+)dt\\
&~= \int_0^{s}2 \lambda(t;x_+)dt\\
 &~= \int_0^{s}\lambda(t;(x_1)_+)dt+\int_0^{s}\lambda(t;(x_2)_+)dt\\
 &\stackrel{\eqref{Hineq}}{\geq}  \tau(e(x_1)_+) + \tau(e(x_2)_+)\\
 &~\ge  \tau(e x_1) + \tau(ex_2)\\
 &~= \tau(e(x_1+x_2)).
\end{align*}
Hence,  $\int_0^{s}\lambda(t;(x_i)_+)dt =\tau(ex_i),\; i=1,2$.
 From Lemma \ref{reduction to commutative}, we obtain that $e=E^{x}(\lambda(s; x_+ ),\infty) =E^{x_1}(\lambda(s;(x_1)_+),\infty)=E^{x_2}(\lambda(s;(x_2)_+),\infty).$
Hence,
we have
\begin{align}\label{x012}
E^x(\theta,\infty) = E^x(\lambda(s;x_+),\infty)& = E^{x_1}(\lambda (s;(x_1)_+),\infty)= E^{x_2}(\lambda (s;(x_2)_+ ),\infty)\nonumber\\
&= E^{x_1}(\theta,\infty)=E^ {x_2}(\theta,\infty).
\end{align}
Note that (for convenience, we denote $E^x [\infty,\infty )=E^{x_1}[\infty,\infty )=E^{x_2} [\infty,\infty )=0$)
\begin{align*}&\quad E^x [\lambda(0;x_+),\infty ) = {\bf 1}- E^x (-\infty, \lambda(0;x_+) ) ={\bf 1} -\lim_{\theta \uparrow  \lambda(0;x_+)^-}E^x (-\infty, \theta  ] \\
&\stackrel{\eqref{x012}}{=} {\bf 1}-\lim_{\theta \uparrow  \lambda(0;x)^-}E^{x_1} (-\infty, \theta  ]= {\bf 1}- E^{x_1} (-\infty, \lambda(0;(x_1)_+) ) =E^{x_1} [\lambda(0;(x_1)_+ ),\infty ) \\
&\stackrel{\eqref{x012}}{=}  {\bf 1}-\lim_{\theta \uparrow  \lambda(0;x_+)^-}E^{x_2} (-\infty, \theta  ]= {\bf 1} - E^{x_2} (-\infty, \lambda(0;(x_2)_+ ) ) =E^{x_2} [\lambda(0;(x_2)_+),\infty ),
\end{align*}
which together with \eqref{x012} implies that
$$ x E^x (0,\infty ) = x_1 E^{x_1} (0,\infty )= x_2 E^{x_2}( 0,\infty )  . $$
It follows that $$x_+=(x_1)_+=(x_2)_+.$$
 The same argument shows that
  $$x_-=(x_1)_-=(x_2)_-.$$
\end{proof}

\end{document}